\documentclass[12pt]{amsart}
\usepackage{a4wide}
\usepackage{amsmath,amssymb}
\usepackage{amsrefs}
\usepackage{todonotes,enumerate,ourbib,mathrsfs}
\usepackage{overpic}
\usepackage{mathptmx}

\newtheorem{maintheorem}{Theorem}

\newtheorem{theorem}{Theorem}[section]
\newtheorem{proposition}[theorem]{Proposition}
\newtheorem{corollary}[theorem]{Corollary}
\newtheorem{lemma}[theorem]{Lemma}
\theoremstyle{definition}
\newtheorem{definition}[theorem]{Definition}
\newtheorem{remark}[theorem]{Remark}
\newtheorem{example}[theorem]{Example}

\newtheorem{assumption}[theorem]{Assumption}

\numberwithin{equation}{section}

\newcommand{\id}{\textup{id}}

\newcommand{\II}{\mathrm{II}}
\newtheorem*{AngleHypothesis}{Matching Angle Hypothesis}
\renewcommand{\mathcal}{\mathscr}

\usepackage{xcolor}
\usepackage{hyperref}
\hypersetup{
    colorlinks,
    linkcolor={red!50!black},
    citecolor={blue!50!black},
    urlcolor={blue!80!black}
}

\newcommand{\R}{\mathbb{R}}
\newcommand{\C}{\mathbb{C}}

\newcommand{\tr}{\mathrm{tr}}

\keywords{Scalar curvature, initial data sets, polyhedra, holographic index theorem, spin geometry, embedding into Minkowski spacetime}
\subjclass[2020]{53C20, 53C21, 53C27, 58J20, 58J32}

\author[C.~B\"ar]{Christian B\"ar}
\address{Institut f\"ur Mathematik, Universit\"at Potsdam, 14476 Potsdam, Germany}
\thanks{C.B.\ was supported by DFG-SPP 2026 ``Geometry at Infinity'' and by the University of Augsburg.} 
\email{\href{mailto:cbaer@uni-potsdam.de}{cbaer@uni-potsdam.de}}

\author[S.~Brendle]{Simon Brendle}
\address{Department of Mathematics, Columbia University, New York NY 10027, USA}
\thanks{S.B.\ was supported by the National Science Foundation under grant DMS-2403981 and by the Simons Foundation.} 
\email{\href{simon.brendle@columbia.edu}{simon.brendle@columbia.edu}}

\author[T.-K.~A.~Chow]{Tsz-Kiu Aaron Chow}
\address{Department of Mathematics, Hong Kong University of Science and Technology, Hong Kong S.A.R., China}
\thanks{T.-K.A.C.\ was supported by the Croucher Foundation.} 
\email{\href{chowtka@ust.hk}{chowtka@ust.hk}}

\author[B.~Hanke]{Bernhard Hanke}
\address{Institut f\"ur Mathematik, Universit\"at Augsburg, 86135 Augsburg, Germany}
\thanks{B.H.\ was supported by DFG-SPP 2026 ``Geometry at Infinity'', by Columbia University,  by the NYU Courant Institute, by the Isaac Newton Institute and by the University of Potsdam} 
\email{\href{mailto:hanke@math.uni-augsburg.de}{hanke@math.uni-augsburg.de}}

\begin{document}

\title[Rigidity results for initial data sets]{Rigidity results for initial data sets satisfying the dominant energy condition}
\begin{abstract} Our work proves rigidity theorems for initial data sets associated with compact smooth spin manifolds with boundary and with compact convex polytopes, subject to the dominant energy condition.
For manifolds with smooth boundary, this is based on the solution of a boundary value problem for Dirac operators.
For convex polytopes we use approximations by manifolds with smooth boundary.
\end{abstract}

\maketitle

\section{Introduction}

An \emph{initial data set} is a triplet $(M,g,q)$, where $M$ is a manifold, $g$ is a Riemannian metric on $M$, and $q$ is a symmetric $(0,2)$-tensor on $M$. 
We denote the scalar curvature of $g$ by $R$.
We define 
\[\mu = \tfrac{1}{2} \, R + \tfrac{1}{2} \, \tr(q)^2 - \tfrac{1}{2} \, |q|^2\] 
and 
\[J = \text{\rm div}(q) - \nabla \tr(q).\] 
We say that $(M,g,q)$ satisfies the \emph{dominant energy condition} if $\mu\geq |J|$ at each point in $M$.

\begin{example} \label{ex:hypersurf}
Let $M$ be a spacelike hypersurface in the Minkowski spacetime $\R^{n,1}$. 
Let $g$ and $q$ denote the induced metric and second fundamental form on $M$, respectively. 
Then the initial data set $(M, g, q)$ satisfies $\mu=0$ and $J = 0$ by the Gauss and Codazzi-Mainardi equations, respectively. 
In  particular,  $(M,g,q)$ satisfies the dominant energy condition. 
\end{example}

The study of initial data sets under the dominant energy condition can be traced back to Schoen and Yau's work \cite{Schoen-Yau}, in which they proved the spacetime positive energy theorem using minimal surface techniques in dimension 3. 
Eichmair, Huang, Lee, and Schoen \cite{Eichmair-Huang-Lee-Schoen} later extended this work by proving a more general spacetime positive mass theorem up to dimension 7. 
Huang and Lee \cite{Huang-Lee} subsequently proved the rigidity statement in this theorem.
In a different approach, Witten \cite{Witten} used spinors to prove the spacetime positive mass theorem, a method that generalizes to all higher dimensions for spin manifolds. Parker and Taubes \cite{Parker-Taubes} gave a mathematically rigorous formulation of Witten's argument. 
Chru\'sciel and Maerten \cite{Chrusciel-Maerten} later provided a rigorous proof of the desired rigidity statement in all dimensions for spin manifolds, building on Beig and Chru\'sciel's proof \cite{Beig-Chrusciel} in dimension~3. 
Lee's book \cite{Lee} offers a comprehensive exposition on the positive mass theorem.

Instead of noncompact asymptotically flat manifolds, we consider compact manifolds with boundary. 
The condition of asymptotic flatness is replaced by suitable boundary conditions.
In the interior, we still demand the dominant energy condition.
First rigidity results have been obtained in this context by Eichmair, Galloway and Mendes in \cite{Eichmair-Galloway-Mendes} and by Gl\"ockle in \cite{Gloe}.

We consider the following class of manifolds with boundary, specializing Example~\ref{ex:hypersurf}.

\begin{example} \label{ex:model}
Let $K \subset \R^n$ be a compact convex domain with smooth boundary $\partial K$ and exterior normal $N \colon \partial K \to S^{n-1}$.
Note that $N$ has degree $1$.
Let  $f \colon K \to \R$ be a smooth map satisfying the following conditions:
\begin{enumerate}[(a)]
    \item $|df| < 1$ on $K$, 
    \item $f = 0$ along $\partial K$, 
    \item \label{boundary} At each point on $\partial K$, we have $df(N)=0$ or $dN=0$.
\end{enumerate}
Let $M := \{ (x , f(x)) \colon x \in K \} \subset \R^{n,1}$ denote the graph of $f$. 
Then $M$ is a compact smooth spacelike hypersurface  with boundary $\Sigma = \partial K \times \{0\} \subset \R^n \times \{0\} \subset \R^{n,1}$, see Figures~\ref{fig.1} and~\ref{fig.2}.
\begin{figure}[h]
\begin{overpic}[scale=.3,]{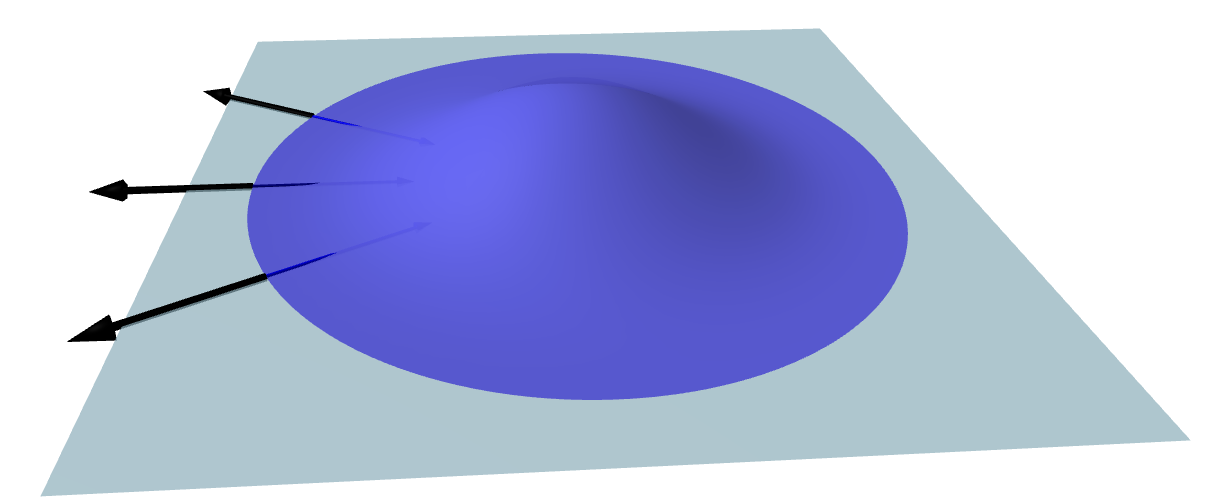}
\put(13,7){$\R^n$}
\put(60,20){\textcolor{white}{$M$}}
\put(4,24.5){$N$}
\end{overpic}
\caption{The case $df(N)=0$ and $dN\ne0$}
\label{fig.1}
\end{figure}
\begin{figure}[h]
\begin{overpic}[scale=.2,]{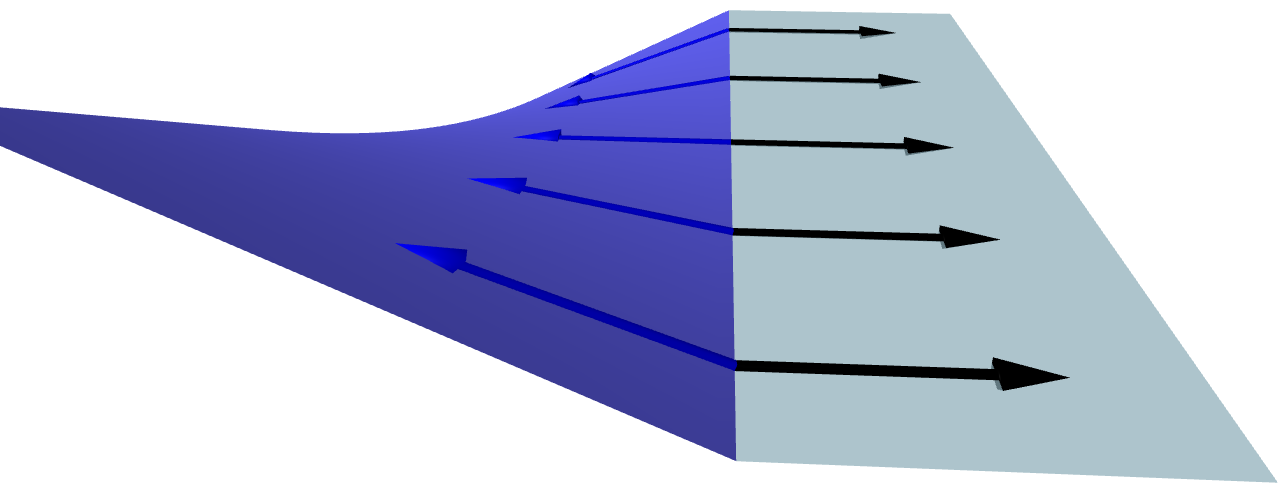}
\put(66,13){$\R^n$}
\put(17,22){\textcolor{white}{$M$}}
\put(74,24){$N$}
\end{overpic}
\caption{The case $df(N)\ne0$ and $dN=0$}
\label{fig.2}
\end{figure}
Let $\II$ denote the vector-valued second fundamental form of $\Sigma$, viewed as a submanifold of $\R^{n,1}$ of codimension $2$. Then 
\[\II (X,Y) = -\langle dN(X),Y \rangle \, N\] 
for all tangent vectors $X,Y \in T\Sigma$.

We denote by $T$ the future-oriented normal vector field to $M$ in $\R^{n,1}$, where $T$ is normalized so that $\langle T,T \rangle_{\R^{n,1}} = -1$. Let $g$ denote the induced Riemannian metric and let $q$ denote the second fundamental form of $M$ with respect to $T$. In other words, $q(X,Y) = \langle \bar{D}_X T,Y \rangle_{\R^{n,1}}$ for all $X,Y \in TM$, where $\bar{D}$ denotes the standard flat connection on $\R^{n,1}$. The initial data set $(M,g,q)$ satisfies the dominant energy condition, cf.\ Example~\ref{ex:hypersurf}. 
At each point on $\Sigma$, we denote by $\nu$ the outward-pointing unit normal vector to $\Sigma$ in $M$. Moreover, we denote by $h$ the second fundamental form of $\Sigma$, viewed as a hypersurface in $M$.
In other words, $h(X,Y) = \langle D_X \nu, Y \rangle_{\R^{n,1}}$ for $X, Y \in T\Sigma$, where $D$ denotes the Levi-Civita connection on $(M,g)$.
We will use these sign conventions for $q$ and $h$ throughout our paper.

With this understood, 
\[h(X,Y) \, \nu - q(X,Y) \, T = -\II (X,Y) = \langle dN(X),Y \rangle \, N\] 
for all tangent vectors $X,Y \in T\Sigma$. 
As $\langle \nu,\nu \rangle_{\R^{n,1}} = 1$, $\langle T,T \rangle_{\R^{n,1}} = -1$, and $\langle \nu,T \rangle_{\R^{n,1}} = 0$, it follows that  
\[h(X,Y) = \langle dN(X),Y \rangle \, \langle N,\nu \rangle_{\R^{n,1}}\] 
and 
\[q(X,Y) = \langle dN(X),Y \rangle \, \langle N,T \rangle_{\R^{n,1}}\] 
for all tangent vectors $X,Y \in T\Sigma$. 
Assumption~\eqref{boundary} implies that, at each point on $\Sigma$, we either have $dN=0$, or $\langle N,\nu \rangle_{\R^{n,1}} = 1$ and $\langle N,T \rangle_{\R^{n,1}} = 0$. 
Consequently, $h(X,Y) = \langle dN(X),Y \rangle$ and $q(X,Y) = 0$ for all tangent vectors $X,Y \in T\Sigma$. 
In particular, we have $H = \|dN\|_\tr$ and $\tr(q) - q(\nu,\nu) = 0$ at each point on $\Sigma$. 
Here, $H$ denotes the unnormalized mean curvature, i.e. the sum of the principal curvatures of $\Sigma$ in $M$. 
In other words, $H$ is equal to the trace of $h$, considered as an endomorphism field.
Moreover, $\|\cdot\|_\tr$ denotes the trace norm of a linear map, i.e. the sum of all singular values.
\end{example}

Our first result  identifies this class of examples  as an extremal case of initial data sets.

\begin{maintheorem}
\label{rigidity.of.smooth.domains}
Let $(M,g,q)$ be an initial data set.
Assume that $M$ is a compact connected spin manifold, has dimension $n \geq 2$ and nonempty boundary $\partial M = \Sigma$. 
Let $N \colon \Sigma \to S^{n-1}$ be a smooth map. 
We assume that the following conditions are satisfied:
\begin{itemize}
\item $(M,g,q)$ satisfies the dominant energy condition.
\item $H \geq \|dN\|_\tr + |\tr(q) - q(\nu,\nu)|$ holds along $\Sigma$.
\item The map $N \colon \Sigma \to S^{n-1}$ has positive degree.
\end{itemize}
Then $(M,g)$ is isometric to a manifold appearing in Example~\ref{ex:model} such that $q$ is the second fundamental form of $M$ in $\R^{n,1}$. Moreover, the map $N$ agrees, up to an element of $O(n)$, with the exterior normal $\partial K \to S^{n-1}$.
\end{maintheorem}

\begin{remark} We do not assume that $\Sigma$ is connected. 
If $\Sigma$ is not connected, the degree of $N \colon \Sigma \to S^{n-1}$ equals the sum of the mapping degrees of $N$ restricted to the components of $\Sigma$ with their orientations induced from that of $M$. 
\end{remark} 

\begin{remark} The quantity $H - |\tr(q) - q(\nu,\nu)|$ is related to the notion of marginally trapped surfaces in general relativity. 
Recall that a marginally outer trapped surface (MOTS) is a hypersurface in an initial data set $(M,g,q)$ satisfying $H + (\tr(q) - q(\nu,\nu)) = 0$; see \cite{Eichmair-Huang-Lee-Schoen}*{p.~89}.
\end{remark}

\begin{remark}
There is also a variant of Theorem~\ref{rigidity.of.smooth.domains} for $n=1$.
In this case, $M=[a,b]$ is a compact interval, $g=dt\otimes dt$ is the standard metric, and $q=\kappa g$ where $\kappa\colon [a,b]\to\R$ is a given smooth function.
We can choose $N \colon \partial M=\{a,b\}\to S^0=\{-1,1\}$ to be bijective.
Then all quantities $R$, $\mu$, $H$, $J$, $dN$, and $\tr(q) - q(\nu,\nu)$ vanish.
Thus the assumptions of Theorem~\ref{rigidity.of.smooth.domains} are trivially satisfied.

The conclusion is now that there is a spacelike curve $[a,b]\to \R^{1,1}$, para\-metrized by arc-length and with curvature $\kappa$.
Indeed, this can be shown directly by ODE methods.
The analogue for the Euclidean plane $\R^2$ instead of $\R^{1,1}$ is known as the fundamental theorem of planar curve theory.
\end{remark}

For $q = 0$, Theorem~\ref{rigidity.of.smooth.domains} gives the following statement:

\begin{corollary} Let $(M,g)$ be a compact connected Riemannian spin manifold of dimension $n \geq 2$ and nonempty boundary $\partial M = \Sigma$.
Let $N \colon \Sigma \to S^{n-1}$ be a smooth map.
We assume that the following conditions are satisfied:
\begin{itemize}
\item $R \geq 0$ holds on $M$.
\item $H \geq \|dN\|_\tr$ holds along $\Sigma$.
\item The map $N \colon \Sigma \to S^{n-1}$ has positive degree.
\end{itemize}
Then $M$ can be isometrically embedded in $\R^n$ as a compact convex domain with smooth boundary. Moreover, the map $N$ agrees, up to an element of $O(n)$, with the exterior normal $\partial K \to S^{n-1}$.
\end{corollary}

The proof of Theorem~\ref{rigidity.of.smooth.domains} uses a boundary value problem for the Dirac operator. Boundary value problems for the Dirac operator have been extensively studied thanks to the work of H\"ormander \cite{Hormander}, and the first named author and Ballmann \cites{Baer-Ballmann,Baer-Ballmann13}.
In the second part of this paper, we consider initial data sets of the form $(\Omega, g, q)$, where $\Omega$ is a compact convex polytope in $\R^n$ with nonempty interior. 
When $q = 0$, the dominant energy condition simplifies to the nonnegativity of the scalar curvature $R \geq 0$, and the boundary condition reduces to mean convexity $H \geq 0$, allowing us to compare $(\Omega, g)$ with Euclidean polytopes. 
Gromov \cite{Gromov14} initiated the study of scalar curvature comparison for polytopes with nonnegative scalar curvature to explore the notion of scalar curvature lower bounds in low-regularity spaces, and this problem has been extensively investigated; see \cites{BW23,Gromov14, Gromov19, Gromov22, Li17, Li17Err, Li19, Wang-Xie-Yu}. 
For $q = g$, the dominant energy condition reduces to the inequality $R \geq -n(n-1)$. 
Li \cite{Li20} established a polyhedral comparison result for metrics with scalar curvature at least $-n(n-1)$. 
The case of general initial data sets was addressed by Tsang \cite{Tsang}, in the special case when $\Omega$ is a three-dimensional cube.

We will employ the strategy developed by the second named author in \cite{Brendle}, adapting it to initial data sets. We write the given polytope in the form $\Omega = \bigcap_{i \in I} \{ u_i \leq 0 \}$, where $I$ is a finite set and $u_i$ are linear functions on $\R^n$. For each $i \in I$, let $N_i \in S^{n-1}$ and $\nu_i$ denote the outward pointing unit normal vectors to the half-space $\{ u_i \leq 0 \}$ with respect to the Euclidean metric and the metric $g$, respectively. 
We make the following assumption:

\begin{AngleHypothesis} 
If $x\in\partial\Omega$ is a point on the boundary and $i_1, i_2\in I$ satisfy $u_{i_1}(x) = u_{i_2}(x) = 0$, then $\langle \nu_{i_1}(x), \nu_{i_2}(x)\rangle_g = \langle N_{i_1}, N_{i_2}\rangle$. 
Here the inner product $\langle \nu_{i_1}, \nu_{i_2}\rangle_g$ is computed with respect to the metric $g$ and $\langle N_{i_1}, N_{i_2}\rangle$ is the standard Euclidean inner product.
\end{AngleHypothesis}

\begin{maintheorem}
\label{rigidity.of.polytopes}
Assume that $n \geq 2$ is an integer. 
Let $\Omega$ be a compact convex polytope in $\R^n$ with nonempty interior. Let $g$ be a Riemannian metric which is defined on an open set containing $\Omega$. 
Let $q$ be a symmetric $(0,2)$-tensor, which is defined on an open set containing $\Omega$. 
We assume that the following conditions are satisfied:
\begin{itemize}
\item  $(\Omega ,g,q)$ satisfies the dominant energy condition.
\item $H\geq |\tr(q) - q(\nu,\nu)|$ on the boundary faces of $\Omega$.
\item The Matching Angle Hypothesis is satisfied.
\end{itemize}
Then $(\Omega,g,q)$ can be isometrically immersed as a spacelike hypersurface in the Minkowski spacetime $\R^{n,1}$ such that $q$ is the second fundamental form. Moreover, each boundary face of $\Omega$ is totally geodesic as a codimension 2 submanifold in the Minkowski spacetime $\R^{n,1}$.
\end{maintheorem}

\section{A boundary value problem for the Dirac operator in odd dimensions}

Throughout this section, we assume that $M$ is a compact connected spin manifold of dimension $n \geq 2$ with boundary $\partial M= \Sigma$. 
Let $g$ be a Riemannian metric on $M$. 
Suppose that $\nu$ is the unit outward normal field along $\partial M$ and that $N \colon \partial M \to S^{n-1}$ is a given smooth map. 

Let $m = 2^{[\frac{n}{2}]}$. 
Let $\mathcal{S}$ denote the spinor bundle over $M$. 
Note that $\mathcal{S}\to M$ is a complex vector bundle of rank $m$ equipped with a Hermitian inner product and compatible connection $\nabla^{\mathcal{S}}$. 

Let $S_0$ denote the spinor module of the complex Clifford algebra $\C \mathrm{l}(\R^n)$ of $\R^n$.
Both $S_0$ and its dual space $S_0^*$ are $m$-dimensional Hermitian vector spaces.

We define a complex vector bundle $\mathcal{E}$ over $M$ by $\mathcal{E} = \mathcal{S} \otimes S_0^*$.
The bundle $\mathcal{E}$ has rank $m^2$ and inherits a natural Hermitian inner product and compatible connection $\nabla^{\mathcal{E}}$.

\begin{remark}
The fiber of $\mathcal{E}\to M$ over $x\in M$ is given by 
\begin{equation*}
\mathcal{E}|_x = \mathcal{S}|_x \otimes S_0^* = \mathrm{Hom}(S_0, \mathcal{S}|_x).
\label{eq.Tensorproduct=Hom}
\end{equation*}
Thus, a section $s$ of $\mathcal{E}$ can be considered as a field of homomorphisms.
Given an element $\sigma \in S_0$ we can apply $s$ at each point $x\in M$ to $\sigma$ and obtain a section $s\sigma$ of $\mathcal{S}$.
\end{remark}

We denote by $\gamma \colon TM\to \mathrm{End}(\mathcal{S})$ and $\gamma_0 \colon \R^n \to \mathrm{End}(S_0)$ the Clifford multiplication of $(M,g)$ and the one of $\R^n$, respectively.
This gives an induced Clifford multiplication 
\[
\gamma_0^*\colon \R^n \to \mathrm{End}(S_0^*),
\]
where $\gamma_0^*(\xi)$ is defined as the adjoint of $\gamma_0(\xi)$. We define 
\begin{equation}
\chi\colon\mathcal{E}|_\Sigma \to \mathcal{E}|_\Sigma
\quad\mbox{ by }\quad
\chi_x := - \gamma(\nu(x)) \otimes \gamma_0^*(N(x)).
\label{boundary.chirality}
\end{equation}
Using the identity $\gamma(\nu)^2 = \gamma_0^*(N)^2=-\id$, we obtain $\chi^2=\id$. Since both $\gamma(\nu)$ and $\gamma_0^*(N)$ are skew-adjoint, we have $\chi^*=\chi$.

For each $X \in T\Sigma$, the linear map $\gamma(X) \otimes \id$ anticommutes with $\chi$ and is an isomorphism if $X\neq 0$. Therefore the $\pm 1$-eigenspaces of $\chi$ have the same dimension.

Let $\mathcal{D}$ denote the Dirac operator on $M$, acting on sections of $\mathcal{E}$. 
This Dirac operator is formally self\-adjoint. 
By \cite{Baer-Ballmann}*{Corollary~7.23}, the conditions $\chi s=s$ and $\chi s=-s$ both form elliptic boundary conditions for $\mathcal{D}$.
Since $\gamma(\nu) \otimes \id$ commutes with $\chi$, the boundary conditions $\chi s=s$ and $\chi s=-s$ are adjoint to each other, see \cite{Baer-Ballmann}*{Section~7.2}.

We denote by $L^2(M,\mathcal{E})$ the Hilbert space of square integrable sections of $\mathcal{E}$ and by $H^1(M,\mathcal{E})$ the Sobolev space of $L^2$-sections whose distributional derivatives are also square integrable.
If $\mathcal{F}\subset \mathcal{E}|_\Sigma$ is a subbundle of $\mathcal{E}$ along the boundary, we write 
\[
H^1(M,\mathcal{E},\mathcal{F}) := \{u\in H^1(M,\mathcal{E}) : u|_\Sigma \in L^2(\Sigma,\mathcal{F})\}.
\]
By the trace theorem, the restriction map extends uniquely to a bounded linear map 
\[
H^1(M,\mathcal{E})\to L^2(\Sigma,\mathcal{E}).
\]

We decompose the bundle 
\begin{equation}
\mathcal{E}|_\Sigma = \mathcal{F}^+ \oplus \mathcal{F}^-
\label{F+F-}
\end{equation}
where $\mathcal{F}^\pm$ are the eigensubbundles of $\chi$ corresponding to the eigenvalues $\pm1$.

As usual, we denote by $\nu$ the outward unit normal along $\Sigma$.
Let $H$ denote the unnormalized mean curvature.
The sign convention is such that the boundary of a Euclidean ball has positive mean curvature. We define an operator $\mathcal{D}^\Sigma$ acting on sections of $\mathcal{E}|_\Sigma = \mathcal{S}|_\Sigma \otimes S_0^*$ by 
\[\mathcal{D}^\Sigma = \sum_{j=1}^{n-1} (\gamma(\nu) \otimes \id)(\gamma(e_j) \otimes \id) \, \nabla_{e_j}^{\mathcal{E}} + \tfrac{1}{2} H,\] 
where $e_1,\hdots,e_{n-1}$ denotes a local orthonormal frame on $\Sigma$ and $\nabla^{\mathcal{E}}$ denotes the connection on $\mathcal{E} = \mathcal{S} \otimes S_0^*$ defined above. Moreover, we define an operator $\mathcal{A}$ acting on sections of $\mathcal{E}|_\Sigma = \mathcal{S}|_\Sigma \otimes S_0^*$ by 
\[\mathcal{A} = \mathcal{D}^\Sigma + \tfrac{1}{2} \, \chi \circ \Big ( \sum_{j=1}^{n-1} \gamma(e_j) \otimes \gamma_0^*(dN(e_j)) \Big ).\] 
Let $\mathcal{S}_0$ denote the trivial bundle over $S^{n-1}$ with fiber $S_0$, and let $\mathcal{S}_0^*$ denote the dual bundle over $S^{n-1}$ with fiber $S_0^*$. 
With this understood, we may write $\mathcal{E}|_\Sigma = \mathcal{S}|_\Sigma \otimes N^* \mathcal{S}_0^*$. 
By \cite{Baer96}*{Proposition~2.2}, $\mathcal{D}^\Sigma$ can be identified with a twisted Dirac operator acting on sections of 
\[
\mathcal{E}|_\Sigma = \mathcal{S}|_\Sigma \otimes N^* \mathcal{S}_0^*,
\]
where the twist bundle $N^*\mathcal{S}_0^*$ carries the flat connection. 

\begin{lemma}\label{lemma.Dchianticommutator}
Along $\Sigma$ we have
\[
\mathcal{D}^\Sigma \chi + \chi \mathcal{D}^\Sigma = -\sum_{j=1}^{n-1} \gamma(e_j) \otimes \gamma_0^*(dN(e_j)),
\]
where $e_1,\dots,e_{n-1}$ is a local orthonormal tangent frame to $\Sigma$.
\end{lemma}

\begin{proof}
Recall that $\mathcal{D}^\Sigma$ anticommutes with $\gamma(\nu) \otimes \id$. Hence, we compute for a smooth section $s$ of $\mathcal{E}|_\Sigma$: 
\begin{align*} 
\mathcal{D}^\Sigma\chi s 
&= -\mathcal{D}^\Sigma (\gamma(\nu) \otimes \id)(\id \otimes \gamma_0^*(N))s \\ 
&= (\gamma(\nu) \otimes \id) \mathcal{D}^\Sigma (\id \otimes \gamma_0^*(N))s \\ 
&= -\sum_{j=1}^{n-1} (\gamma(e_j) \otimes \id) \, \nabla_{e_j}^{\mathcal{E}}( (\id \otimes \gamma_0^*(N))s) + \tfrac{1}{2} H \, (\gamma(\nu) \otimes \id) (\id \otimes \gamma_0^*(N))s \\ 
&= -\sum_{j=1}^{n-1} (\gamma(e_j) \otimes \id) (\id \otimes \gamma_0^*(N)) \, \nabla_{e_j}^{\mathcal{E}} s + \tfrac{1}{2} H \, (\gamma(\nu) \otimes \id) (\id \otimes \gamma_0^*(N))s \\ 
&\quad - \sum_{j=1}^{n-1} (\gamma(e_j) \otimes \id) (\id \otimes \gamma_0^*(dN(e_j)))s \\ 
&= (\gamma(\nu) \otimes \id)(\id \otimes \gamma_0^*(N)) \, \mathcal{D}^\Sigma s - \sum_{j=1}^{n-1} (\gamma(e_j) \otimes \id) (\id \otimes \gamma_0^*(dN(e_j)))s \\ 
&= -\chi \mathcal{D}^\Sigma s - \sum_{j=1}^{n-1} (\gamma(e_j) \otimes \id) (\id \otimes \gamma_0^*(dN(e_j)))s.
\qedhere
\end{align*} 
\end{proof}

\begin{lemma}\label{lemma.Achianticommutator}
Along $\Sigma$ we have $\mathcal{A} \chi + \chi \mathcal{A} = 0$.
\end{lemma}

\begin{proof}
This follows from Lemma~\ref{lemma.Dchianticommutator}.
\end{proof}

The following index computation generalizes the results in \cite{Brendle} and plays a key role in our arguments.

\begin{proposition} 
\label{index}
Suppose that $M$ has odd dimension $n \geq 3$. 
Then the operator 
\[
\mathcal{D} \colon H^1(M,\mathcal{E},\mathcal{F}^+) \to L^2(M,\mathcal{E})
\]
is a Fredholm operator.
Its Fredholm index equals the degree of $N \colon \Sigma \to S^{n-1}$. 
\end{proposition}

\begin{proof}
Note that $\mathcal{D}$ and $\mathcal{A}$ are formally self-adjoint, and $\mathcal{A}$ is an adapted boundary operator for $\mathcal{D}$. It follows from Lemma~\ref{lemma.Achianticommutator} that $\mathcal{A}$ is an odd operator in the sense that it interchanges the bundles $\mathcal{F}^+$ and $\mathcal{F}^-$. Moreover, the involution $i\gamma(\nu) \otimes \id$ preserves the splitting~\eqref{F+F-} and anticommutes with $\mathcal{A}$. Therefore, the holographic index theorem \cite{Baer-Brendle-Hanke-Wang}*{Theorem B.1} applies and tells us that the operator
\[
\mathcal{D} \colon H^1(M,\mathcal{E},\mathcal{F}^+) \to L^2(M,\mathcal{E})
\]
is Fredholm and 
\[
\mathsf{ind}(\mathcal{D} \colon H^1(M,\mathcal{E},\mathcal{F}^+) \to L^2(M,\mathcal{E}))
=
\tfrac12\, \mathsf{ind}(\mathcal{A} \colon H^1(\Sigma,\mathcal{F}^+) \to L^2(\Sigma,\mathcal{F}^-)).\] 
It remains to compute the index of the operator $\mathcal{A} \colon H^1(\Sigma,\mathcal{F}^+) \to L^2(\Sigma,\mathcal{F}^-)$. 
Since $n$ is odd, the restriction $\mathcal{S}|_\Sigma$ can be canonically identified with the spinor bundle of $\Sigma$. 
The field of involutions $i\gamma(\nu)$ gives a decomposition $\mathcal{S}|_\Sigma = \mathcal{S}_\Sigma^+ \oplus \mathcal{S}_\Sigma^-$, where $\mathcal{S}_\Sigma^\pm$ are the eigensubbundles corresponding to the eigenvalues $\pm 1$. Moreover, we may decompose the bundle $\mathcal{S}_0^*$ as 
\begin{equation}
\mathcal{S}_0^* = \mathcal{S}_0^{*,+} \oplus \mathcal{S}_0^{*,-},
\label{eq:S0splitting}
\end{equation}
where, for each point $v \in S^{n-1}$, $\mathcal{S}_0^{*,\pm}$ denote the eigensubbundles of $i\gamma(v)$ corresponding to the eigenvalues $\pm 1$. 
This gives a decomposition 
\begin{equation*} 
N^* \mathcal{S}_0^* = N^* \mathcal{S}_0^{*,+} \oplus N^* \mathcal{S}_0^{*,-}, 
\label{eq:S0splitting.pullback}
\end{equation*}
where $N^* \mathcal{S}_0^{*,\pm}$ are the eigensubbundles of $i \gamma_0^*(N)$ corresponding to the eigenvalues $\pm 1$. We then have
\begin{align*}
\mathcal{F}^+ 
&= 
(\mathcal{S}_\Sigma^+ \otimes N^*\mathcal{S}_0^{*,+}) \oplus (\mathcal{S}_\Sigma^- \otimes N^*\mathcal{S}_0^{*,-}), \\
\mathcal{F}^-
&= 
(\mathcal{S}_\Sigma^+ \otimes N^*\mathcal{S}_0^{*,-}) \oplus (\mathcal{S}_\Sigma^- \otimes N^*\mathcal{S}_0^{*,+}).
\end{align*}
Note that even though $\mathcal{S}_0^*$ is a trivial bundle, the two subbundles $\mathcal{S}_0^{*,+}$ and $\mathcal{S}_0^{*,-}$ are not because the decomposition depends on the base point. The splitting~\eqref{eq:S0splitting} is not parallel with respect to the canonical flat connection on $\mathcal{S}_0^*$, but the Levi-Civita connection of $S^{n-1}$ induces a connection on the dual of its spinor bundle for which it is. We equip $\mathcal{S}_0^{*,\pm}$ with these latter connections and $N^*\mathcal{S}_0^{*,\pm}$ with the corresponding pull-back connections. 

Recall that $\mathcal{A}$ maps sections of 
\[
\mathcal{F}^+ = (\mathcal{S}_\Sigma^+ \otimes N^*\mathcal{S}_0^{*,+}) \oplus (\mathcal{S}_\Sigma^- \otimes N^*\mathcal{S}_0^{*,-})
\]
to sections of $\mathcal{F}^- = (\mathcal{S}_\Sigma^+ \otimes N^*\mathcal{S}_0^{*,-}) \oplus (\mathcal{S}_\Sigma^- \otimes N^*\mathcal{S}_0^{*,+})$. Since the involution $i\gamma(\nu) \otimes \id$ preserves the splitting~\eqref{F+F-} and anticommutes with $\mathcal{A}$, it follows that $\mathcal{A}$ maps sections of $\mathcal{S}_\Sigma^+ \otimes N^*\mathcal{S}_0^{*,+}$ to sections of $\mathcal{S}_\Sigma^- \otimes N^*\mathcal{S}_0^{*,+}$, and $\mathcal{A}$ maps sections of $\mathcal{S}_\Sigma^- \otimes N^*\mathcal{S}_0^{*,-}$ to sections of $\mathcal{S}_\Sigma^+ \otimes N^*\mathcal{S}_0^{*,-}$. Therefore, the operator $\mathcal{A} \colon \mathcal{F}^+ \to \mathcal{F}^-$ can be written in the form 
\[\begin{bmatrix}
0 & \mathcal{D}^\Sigma_- \\
\mathcal{D}^\Sigma_+ & 0
\end{bmatrix},\]
where 
\begin{align*}
\mathcal{D}^\Sigma_+
&\colon 
H^1(\Sigma,\mathcal{S}_\Sigma^+ \otimes N^*\mathcal{S}_0^{*,+}) \to L^2(\Sigma,\mathcal{S}_\Sigma^- \otimes N^*\mathcal{S}_0^{*,+}) ,\\
\mathcal{D}^\Sigma_-
&\colon 
H^1(\Sigma,\mathcal{S}_\Sigma^- \otimes N^*\mathcal{S}_0^{*,-}) \to L^2(\Sigma,\mathcal{S}_\Sigma^+ \otimes N^*\mathcal{S}_0^{*,-})
\end{align*}
are Dirac-type operators. Thus, 
\[\mathsf{ind}(\mathcal{A} \colon H^1(\Sigma,\mathcal{F}^+) \to L^2(\Sigma,\mathcal{F}^-)) = \mathsf{ind}(\mathcal{D}^\Sigma_+) + \mathsf{ind}(\mathcal{D}^\Sigma_-).\] 
We may compute $\mathsf{ind}(\mathcal{D}^\Sigma_+)$ and $\mathsf{ind}(\mathcal{D}^\Sigma_-)$ using the Atiyah-Singer index theorem. To that end, we denote by $\hat{\mathsf{A}}(\Sigma)$ the $\hat{\mathsf{A}}$-form of the tangent bundle of $\Sigma$, and by $\mathsf{ch}(\mathcal{S}_0^{*,\pm})$ the Chern character form of the bundle $\mathcal{S}_0^{*,\pm}$.
The lower index $k$ in $\hat{\mathsf{A}}(\Sigma)_k$ indicates the homogeneous part of degree $k$. We use the analogous notation for the homogeneous parts of the Chern character.
The Atiyah-Singer index theorem gives 
\begin{align*}
\mathsf{ind}(\mathcal{D}^\Sigma_+)
&=
\int_\Sigma \hat{\mathsf{A}}(\Sigma) \wedge N^*\mathsf{ch}(\mathcal{S}_0^{*,+}) \\
&=
\int_\Sigma \hat{\mathsf{A}}(\Sigma)_0 \wedge N^*\mathsf{ch}(\mathcal{S}_0^{*,+})_{n-1} 
+ \int_\Sigma \hat{\mathsf{A}}(\Sigma)_{n-1} \wedge N^*\mathsf{ch}(\mathcal{S}_0^{*,+})_{0} \\
&=
\int_\Sigma  N^*\mathsf{ch}(\mathcal{S}_0^{*,+})_{n-1} 
+ \mathsf{rank}(\mathcal{S}_0^{*,+}) \int_\Sigma \hat{\mathsf{A}}(\Sigma)_{n-1} \\
&=
\mathsf{deg}(N)\int_{S^{n-1}}  \mathsf{ch}(\mathcal{S}_0^{*,+})_{n-1} 
+ \frac{m}{2} \int_\Sigma \hat{\mathsf{A}}(\Sigma)_{n-1} 
\end{align*}
and, similarly,
\begin{align*}
-\mathsf{ind}(\mathcal{D}^\Sigma_-)
&=
\mathsf{deg}(N)\int_{S^{n-1}}  \mathsf{ch}(\mathcal{S}_0^{*,-})_{n-1} + \frac{m}{2} \int_\Sigma \hat{\mathsf{A}}(\Sigma)_{n-1} .
\end{align*}
Therefore,
\begin{align*}
\mathsf{ind}(\mathcal{A} \colon & H^1(\Sigma,\mathcal{F}^+) \to L^2(\Sigma,\mathcal{F}^-)) \\ 
&=
\mathsf{ind}(\mathcal{D}^\Sigma_+) + \mathsf{ind}(\mathcal{D}^\Sigma_-) \\
&=
\mathsf{deg}(N)\int_{S^{n-1}}  (\mathsf{ch}(\mathcal{S}_0^{*,+}) - \mathsf{ch}(\mathcal{S}_0^{*,-}))_{n-1} \\
&=
2\,\mathsf{deg}(N).
\end{align*}
Putting these facts together, the assertion follows.
\end{proof}

\section{Spin geometry for initial data sets} 

Let $M$ be a compact spin manifold of dimension $n \geq 2$ with boundary $\partial M = \Sigma$. 
Let $g$ be a Riemannian metric on $M$, let $q$ be a symmetric $(0,2)$-tensor field on $M$ and let $N \colon \Sigma \to S^{n-1}$ be a smooth map. 
Let $\chi$ be defined as in \eqref{boundary.chirality}. 
We will denote the self-adjoint endomorphism field corresponding to $q$ also by $q\colon TM\to TM$, i.e.\ $q(X,Y) = g(q(X),Y)$.
Furthermore, we put
\begin{equation}
\mu := \tfrac{1}{2} \, R + \tfrac{1}{2} \, \tr(q)^2 - \tfrac{1}{2} \, |q|^2
\label{def.energy}
\end{equation}
and 
\begin{equation}
J := \text{\rm div}(q) - \nabla \tr(q)
\label{def.momentum}
\end{equation}
where $R$ is the scalar curvature of $g$.
In the first step, we use $q$ to define modified connections on $\mathcal{S}$ and on $\mathcal{E}$.

\begin{definition} 
We define a connection $\nabla^{\mathcal{S},q}$ on $\mathcal{S}$ by 
\begin{equation*} 
\label{modified.connection.on.S}
\nabla_X^{\mathcal{S},q} := \nabla_X^{\mathcal{S}} + \tfrac{i}{2} \gamma(q(X)) 
\end{equation*} 
on $\mathcal{S}$. We next define a homomorphism field $Q\colon \mathcal{E} \to T^*M\otimes\mathcal{E}$ by setting
\begin{equation*}
\label{def.Q}
(Qs)(X) := (\gamma(q(X)) \otimes \id)s
\end{equation*}
for every section $s$ of $\mathcal{E}$ and every tangent vector $X$.
With this understood, we define a connection $\nabla^{\mathcal{E},q}$ on $\mathcal{E}$ by
\begin{equation*}
\label{modified.connection.on.E}
\nabla_X^{\mathcal{E},q} s := \nabla_X^{\mathcal{E}} s + \tfrac{i}{2} (Qs)(X).
\end{equation*} 
Note that the connection $\nabla^{\mathcal{E},q}$ on $\mathcal{E}$ is the tensor product connection of $\nabla^{\mathcal{S},q}$ with the flat connection on the trivial bundle over $M$ with fiber $S_0^*$.
\end{definition}

In order to derive a modified Weitzenb\"ock formula involving $\nabla^{\mathcal{E},q}$ we need some preparation.

\begin{lemma}\label{lemma.Q}
Let $Q^*\colon T^*M\otimes\mathcal{E} \to \mathcal{E}$ denote the pointwise adjoint of $Q$. Then the following statements hold:
\begin{enumerate}[(a)]
\item \label{lemma.Q.1}
$Q^*Q = |q|^2$,
\item \label{lemma.Q.2}
$Q^*\nabla^{\mathcal{E}} - (\nabla^{\mathcal{E}})^*Q = \gamma(\text{\rm div}(q)) \otimes \id$.
\end{enumerate}
\end{lemma}

\begin{proof}
Let $p\in M$ and let $e_1,\dots,e_n$ be an orthonormal basis of $T_pM$.
We denote the dual basis of $T^*_pM$ by $e^1,\dots,e^n$.
For $s\in \mathcal{E}|_p$ and $\psi\in T^*_pM\otimes\mathcal{E}|_p$, we compute 
\begin{align*}
\langle Q^*\psi,s\rangle
&=
\langle \psi,Qs\rangle
=
\Big\langle \psi,\sum_j e^j\otimes (\gamma(q(e_j)) \otimes \id)s \Big\rangle \\
&=
\sum_j \langle \psi(e_j),(\gamma(q(e_j)) \otimes \id)s\rangle \\
&=
-\sum_j \langle (\gamma(q(e_j)) \otimes \id)\psi(e_j),s\rangle.
\end{align*}
Therefore,
\[
Q^*\psi = -\sum_j (\gamma(q(e_j)) \otimes \id)\psi(e_j).
\]
Consequently,
\begin{align*}
Q^*Qs
&=
Q^* \Big( \sum_j e^j\otimes (\gamma(q(e_j)) \otimes \id)s \Big) \\
&=
- \sum_j (\gamma(q(e_j)) \otimes \id)(\gamma(q(e_j)) \otimes \id)s \\
&=
\sum_j |q(e_j)|^2 s
=
|q|^2 s.
\end{align*}
This proves \eqref{lemma.Q.1}.
To show \eqref{lemma.Q.2}, we extend the orthonormal basis to a neighborhood of $p$ such that $\nabla e_j = 0$ at $p$ for $j=1,\hdots,n$.
Let $s$ be a smooth section of $\mathcal{E}$, defined near $p$.
Then we compute at $p$:
\begin{align*}
(\nabla^{\mathcal{E}})^*Qs
&= -\sum_j \nabla_{e_j}^{\mathcal{E}} \big ( (\gamma(q(e_j)) \otimes \id)s \big) \\
&=
-\sum_j (\gamma(q(e_j)) \otimes \id)\nabla_{e_j}^{\mathcal{E}} s - \sum_j (\gamma((D_{e_j}q)(e_j)) \otimes \id)s \\
&=
Q^*\nabla^{\mathcal{E}} s - (\gamma(\text{\rm div}(q)) \otimes \id)s.
\qedhere
\end{align*}
\end{proof}
The well-known Weitzenb\"ock formula gives $\mathcal{D}^2 = (\nabla^{\mathcal{E}})^* \nabla^{\mathcal{E}} + \tfrac{1}{4}R$.
We find a modification, taking $q$ into account.

\begin{proposition}[Modified Weitzenb\"ock formula]
We have the following operator identity:
\[
(\mathcal{D}+\tfrac{i}{2}\tr(q))(\mathcal{D}-\tfrac{i}{2}\tr(q))
=
(\nabla^{\mathcal{E},q})^*\nabla^{\mathcal{E},q} + \tfrac{1}{2}\mu + \tfrac{i}{2}\gamma(J) \otimes \id.
\]
\end{proposition}

\begin{proof}
We compute, using Lemma~\ref{lemma.Q}:
\begin{align*}
(\nabla^{\mathcal{E},q})^*\nabla^{\mathcal{E},q}
&=
((\nabla^{\mathcal{E}})^* -\tfrac{i}{2}Q^*)(\nabla^{\mathcal{E}} + \tfrac{i}{2}Q) \\
&=
(\nabla^{\mathcal{E}})^*\nabla^{\mathcal{E}} -\tfrac{i}{2}(Q^*\nabla^{\mathcal{E}} - (\nabla^{\mathcal{E}})^*Q) + \tfrac{1}{4}Q^*Q \\
&=
(\nabla^{\mathcal{E}})^*\nabla^{\mathcal{E}} -\tfrac{i}{2}\gamma(\text{\rm div}(q)) \otimes \id + \tfrac{1}{4}|q|^2 .
\end{align*}
Furthermore,
\begin{align*}
(\mathcal{D}+\tfrac{i}{2}\tr(q))(\mathcal{D}-\tfrac{i}{2}\tr(q))
&=
\mathcal{D}^2 - \tfrac{i}{2}[\mathcal{D},\tr(q)] + \tfrac{1}{4}\tr(q)^2 \\
&=
\mathcal{D}^2 - \tfrac{i}{2}\gamma(\nabla\tr(q)) \otimes \id + \tfrac{1}{4}\tr(q)^2 .
\end{align*}
These two equations together with the standard Weitzenb\"ock formula yield
\begin{align*}
(\mathcal{D}+\tfrac{i}{2}\tr(q))(\mathcal{D}&-\tfrac{i}{2}\tr(q)) - (\nabla^{\mathcal{E},q})^*\nabla^{\mathcal{E},q} \\
&=
\tfrac{1}{4}R + \tfrac{i}{2}\gamma(\text{\rm div}(q)-\nabla\tr(q)) \otimes \id + \tfrac{1}{4}(\tr(q)^2-|q|^2) \\
&=\tfrac{1}{2}\mu + \tfrac{i}{2}\gamma(J) \otimes \id .
\qedhere
\end{align*}
\end{proof}

\begin{corollary}\label{cor.Weitzenboeck}
For all $s\in H^1(M,\mathcal{E})$ we have
\begin{align*}
-\int_M |(\mathcal{D}-\tfrac{i}{2}\tr(q))s|^2 &+ \int_M |\nabla^{\mathcal{E},q}s|^2 + \tfrac{1}{2}\int_M \langle (\mu+i\gamma(J) \otimes \id)s,s \rangle \\
&=
\int_\Sigma \langle (\mathcal{D}^\Sigma - \tfrac{1}{2}H +\tfrac{i}{2}(\gamma(q(\nu)-\tr(q)\nu) \otimes \id ))s,s \rangle .
\end{align*}
\end{corollary}

\begin{proof}
Note that 
\[
(\gamma(\nu) \otimes \id)\mathcal{D}s = \mathcal{D}^\Sigma s - \nabla_\nu^{\mathcal{E}} s - \tfrac{1}{2} H s 
\]
along $\Sigma$. Using this identity and integration by parts we find
\begin{align*}
\int_M &\langle(\mathcal{D}+\tfrac{i}{2}\tr(q))(\mathcal{D}-\tfrac{i}{2}\tr(q)) s , s \rangle \\
&=
\int_M |(\mathcal{D}-\tfrac{i}{2}\tr(q))s|^2 + \int_\Sigma \langle (\gamma(\nu) \otimes \id)(\mathcal{D}-\tfrac{i}{2}\tr(q)) s , s \rangle \\
&=
\int_M |(\mathcal{D}-\tfrac{i}{2}\tr(q))s|^2 + \int_\Sigma \big\langle \big( \mathcal{D}^\Sigma-\nabla_\nu^{\mathcal{E}} - \tfrac{1}{2}H-\tfrac{i}{2}\tr(q)(\gamma(\nu) \otimes \id)\big) s , s \big\rangle .
\end{align*}
Similarly, 
\begin{align*}
\int_M \langle (\nabla^{\mathcal{E},q})^*\nabla^{\mathcal{E},q} s,s\rangle
&=
\int_M |\nabla^{\mathcal{E},q}s|^2 - \int_\Sigma \langle \nabla_\nu^{\mathcal{E},q} s, s\rangle \\
&=
\int_M |\nabla^{\mathcal{E},q}s|^2 - \int_\Sigma \langle (\nabla_\nu^{\mathcal{E}} + \tfrac{i}{2}\gamma(q(\nu)) \otimes \id) s, s\rangle .
\end{align*}
Substituting these equations into the modified Weitzenb\"ock formula concludes the proof.
\end{proof}

\begin{lemma}\label{lemma.estimate.dN}
We have 
\[\Big|\Big\langle\sum_{j=1}^{n-1} (\gamma(e_j) \otimes \gamma_0^*(dN(e_j)))s,s\Big\rangle\Big| \leq \|dN\|_\tr \, |s|^2\] 
at each point on $\Sigma$. Here, $e_1,\dots,e_{n-1}$ denotes a local orthonormal frame on $\Sigma$.
\end{lemma}

\begin{proof}
Let us fix a point $p \in \Sigma$. We choose orthonormal bases $e_1,\dots,e_{n-1}$ of $T_p\Sigma$ and $e^0_1,\dots,e^0_{n-1}$ of $T_{N(p)} S^{n-1}$ such that $dN(e_j)=\lambda_j e^0_j$, where $\lambda_j\geq 0$ denote the singular values of $dN$.
Then
\begin{align*}
\Big|\Big\langle\sum_{j=1}^{n-1} (\gamma(e_j) \otimes \gamma_0^*(dN(e_j)))s,s\Big\rangle\Big|  
&=
\Big| \sum_{j=1}^{n-1} \lambda_j \, \langle (\gamma(e_j) \otimes \gamma_0^*(e^0_j))s,s \rangle \Big|  \\
&\leq
\sum_{j=1}^{n-1}\lambda_j \, |(\gamma(e_j) \otimes \gamma_0^*(e^0_j))s| \, |s| \\ 
&=
\sum_{j=1}^{n-1}\lambda_j \, |s|^2. 
\qedhere
\end{align*}
\end{proof}

\begin{proposition}\label{integral.formula}
Suppose that $s\in H^1(M,\mathcal{E})$ satisfies the boundary condition $\chi s = \varepsilon s$ along $\Sigma$, where $\varepsilon \colon \Sigma\to\{-1,1\}$ is locally constant.
Then we have the estimate
\begin{align*}
-\int_M |(\mathcal{D}-\tfrac{i}{2}\tr(q))s|^2 + \int_M& |\nabla^{\mathcal{E},q}s|^2 + \tfrac{1}{2}\int_M \langle (\mu+i\gamma(J) \otimes \id)s,s \rangle \\
&\le
\tfrac{1}{2}\int_\Sigma (|q(\nu,\nu)-\tr(q)| + \|dN\|_\tr - H)|s|^2  .
\end{align*}
\end{proposition}

\begin{proof}
We first observe that 
\[2 \, \langle \mathcal{D}^\Sigma s,s \rangle = \varepsilon \, \langle \mathcal{D}^\Sigma \chi s,s \rangle + \varepsilon \, \langle \mathcal{D}^\Sigma s,\chi s \rangle = \varepsilon \, \langle \mathcal{D}^\Sigma \chi s + \chi \mathcal{D}^\Sigma s,s \rangle\] 
at each point on $\Sigma$. Using Lemma~\ref{lemma.Dchianticommutator}, we obtain 
\[2 \, |\langle \mathcal{D}^\Sigma s,s \rangle| = \Big | \Big \langle \sum_{j=1}^{n-1} (\gamma(e_j) \otimes \gamma_0^*(dN(e_j)))s,s \Big \rangle \Big |\] 
at each point on $\Sigma$. 
Lemma~\ref{lemma.estimate.dN} implies 
\begin{equation} 
\label{eq.basic.est1}
2 \, |\langle \mathcal{D}^\Sigma s,s \rangle| \leq \|dN\|_\tr \, |s|^2 
\end{equation}
at each point on $\Sigma$. Moreover, for each $X\in T\Sigma$ we have 
\begin{align*} 
2 \, \langle (\gamma(X) \otimes \id)s,s \rangle 
&= \varepsilon \, \langle (\gamma(X) \otimes \id)\chi s,s \rangle + \varepsilon \, \langle (\gamma(X) \otimes \id)s,\chi s \rangle \\ 
&= \varepsilon \, \langle (\gamma(X) \otimes \id)\chi s + \chi(\gamma(X) \otimes \id)s,s \rangle \\ 
&= 0.
\end{align*} 
Putting $X = q(\nu)^{\text{\rm tan}}$, the part of $q(\nu)$ tangential to $\Sigma$, gives 
\[
\langle (\gamma(q(\nu)) \otimes \id)s,s\rangle = q(\nu,\nu) \, \langle (\gamma(\nu) \otimes \id)s,s\rangle.
\]
We conclude
\begin{align}
\big| \langle (\gamma(q(\nu)-\tr(q)\nu) \otimes \id )s,s \rangle \big|
&=
\big| (q(\nu,\nu)-\tr(q)) \, \langle (\gamma(\nu) \otimes \id )s,s \rangle \big| \notag\\
&\leq 
|q(\nu,\nu)-\tr(q)| \, |s|^2 .
\label{eq.basic.est2}
\end{align}
Using \eqref{eq.basic.est1} and \eqref{eq.basic.est2} together with Corollary~\ref{cor.Weitzenboeck}, the assertion follows.
\end{proof}

\begin{proposition} 
\label{proposition.existence.n.odd}
Suppose that $M$ has odd dimension $n \geq 3$ and the map $N \colon \Sigma \to S^{n-1}$ has positive degree. Then there exists a non-trivial section $s \in C^\infty(M,\mathcal{E})$ such that 
\begin{align*} 
\int_M  | \nabla^{\mathcal{E},q} s |^2 
&\leq 
\tfrac{1}{2} \int_M (|J|-\mu)  |s|^2 
+ \tfrac{1}{2} \int_\Sigma (|\tr(q) - q(\nu,\nu)| + \|dN\|_\tr - H) \, |s|^2. 
\end{align*}
Moreover, $\chi s = s$ at each point on the boundary $\Sigma$.
\end{proposition}

\begin{proof}
By Proposition~\ref{index}, the operator $\mathcal{D}$ subject to the boundary conditions $\chi s = s$ has positive index.
By deformation invariance of the Fredholm index, the operator $\mathcal{D} - \tfrac{i}{2} \tr(q)$ has positive index.
Hence, we can find a non-trivial section $s$ of $\mathcal{E}$ such that $\mathcal{D} s - \tfrac{i}{2} \, \tr(q) \, s = 0$ on $M$ and $\chi s = s$ along $\Sigma$. 
By elliptic regularity, $s$ is smooth.
Using Proposition~\ref{integral.formula}, we obtain 
\begin{align*} 
\int_M  | \nabla^{\mathcal{E},q} s |^2 
&\leq 
-\tfrac{1}{2} \int_M \langle (\mu+i\gamma(J) \otimes \id)s,s \rangle \\
&\quad
+ \tfrac{1}{2} \int_\Sigma (|\tr(q) - q(\nu,\nu)| + \|dN\|_\tr - H) \, |s|^2. 
\end{align*}
From this, the assertion follows easily.
\qedhere
\end{proof}

In the next step, we establish an analogue of Proposition~\ref{proposition.existence.n.odd} in even dimensions. In the remainder of this section, we will assume that $n \geq 2$ is an even integer. 

\begin{lemma}\label{lemma.existence.of.varphi}
We can find a smooth map $\varphi\colon S^{n-1} \times S^1 \to S^n$ with the following properties: 
\begin{itemize}
\item The map $\varphi$ has degree $1$. 
\item $\varphi^* g_{S^n} \leq g_{S^{n-1}} + 4g_{S^1}$. 
\item For each $t \in S^1$, the map $\varphi(\cdot,t) \colon S^{n-1} \to S^n$ is either Lipschitz continuous with a Lipschitz constant strictly less than $1$, or else the map $\varphi(\cdot,t) \colon S^{n-1} \to S^n$ is the standard inclusion $S^{n-1} \hookrightarrow S^n$.
\end{itemize}
\end{lemma} 

\begin{proof} 
Let us fix a point $a \in S^{n-1}$ and a smooth function $\beta \colon [-\pi,\pi] \to [-\pi,\pi]$ with the following properties: 
\begin{itemize}
\item $\beta(t) = -\pi$ for $t \in [-\pi,-\tfrac{7\pi}{8}]$.
\item $\beta(t) = 0$ for $t \in [-\tfrac{\pi}{8},\tfrac{\pi}{8}]$.
\item $\beta(t) = \pi$ for $t \in [\tfrac{7\pi}{8},\pi]$.
\item $0 < \beta'(t) \leq 2$ for $t \in (-\tfrac{7\pi}{8},-\tfrac{\pi}{8}) \cup (\tfrac{\pi}{8},\tfrac{7\pi}{8})$.
\end{itemize}
We consider the map 
\[\varphi \colon S^{n-1} \times [-\pi,\pi] \to S^n, \, (\xi,t) \mapsto \begin{cases} (\sin \beta(t) \, a,\cos \beta(t)) & \text{\rm for $t \in [-\pi,0]$}, \\ (\sin \beta(t) \, \xi,\cos \beta(t)) & \text{\rm for $t \in [0,\pi]$.} \end{cases}\] 
It is easy to see that $\varphi$ induces a smooth map from $S^{n-1}\times S^1$ to $S^n$ with degree $\pm 1$. We next compute 
\[\varphi^* g_{S^n} = \begin{cases} \beta'(t)^2 \, dt \otimes dt & \text{\rm for $t \in [-\pi,0]$}, \\ \sin^2 \beta(t) \, g_{S^{n-1}} + \beta'(t)^2 \, dt \otimes dt & \text{\rm for $t \in [0,\pi]$.} \end{cases}\] 
Since $0 \leq \beta'(t) \leq 2$ for all $t \in [-\pi,\pi]$, it follows that $\varphi^* g_{S^n} \leq g_{S^{n-1}} + 4 \, dt \otimes dt$. 

If $t \in [-\pi,0]$, then the map $\varphi(\cdot,t) \colon S^{n-1} \to S^n$ is constant. If $t \in [0,\pi]$ and $\beta(t) \neq \tfrac{\pi}{2}$, then the map $\varphi(\cdot,t) \colon S^{n-1} \to S^n$ is Lipschitz continuous with Lipschitz constant $\sin \beta(t) < 1$. 
If $t \in [0,\pi]$ and $\beta(t) = \tfrac{\pi}{2}$, the map $\varphi(\cdot,t) \colon S^{n-1} \to S^n$ is the standard inclusion $S^{n-1} \hookrightarrow S^n, \, \xi \mapsto (\xi,0)$. 
\end{proof}

In the following, we fix a smooth map $\varphi\colon S^{n-1}\times S^1 \to S^n$ as in Lemma~\ref{lemma.existence.of.varphi}. 
Moreover, $r$ will denote a parameter which will be chosen sufficiently large. 
We consider the product $\tilde{M} = M \times S^1$, equipped with the metric $\tilde{g} = g + r^2 \, dt \otimes dt$. 
We denote by $\tilde{q}$ the pull-back $(0,2)$-tensor field of $q$ under the canonical projection $\tilde{M} \to M$.
We denote by $\tilde{\Sigma} = \Sigma \times S^1$ the boundary of $\tilde{M}$. We define a map $\tilde{N} \colon \tilde{\Sigma} \to S^n$ by $\tilde{N}(x,t) = \varphi(N(x),t)$ for $x \in \Sigma$ and $t \in S^1$. 
Since the map $\varphi\colon S^{n-1} \times S^1 \to S^n$ is of degree $1$, the degree of the map $\tilde{N} \colon \tilde{\Sigma} \to S^n$ agrees with the degree of the map $N \colon \Sigma \to S^{n-1}$. 

\begin{lemma}
\label{lemma.trace.norm}
The trace norm of $d\tilde{N}$ with respect to the metric $\tilde{g}$ satisfies 
\[
\|d\tilde{N}\|_\tr \leq \|dN\|_\tr + 2r^{-1}.
\]
\end{lemma}

\begin{proof}
By Lemma~\ref{lemma.existence.of.varphi}, we have $\varphi^* g_{S^n} \leq g_{S^{n-1}} + 4g_{S^1}$. 
Using the definition of $\tilde{N}$, we obtain 
\[
\tilde{N}^* g_{S^n} \leq N^* g_{S^{n-1}} + 4g_{S^1}.
\]
Let us denote the singular values of $dN_x \colon T_x \Sigma \to T_{N(x)} S^{n-1}$ by $\lambda_1,\hdots,\lambda_{n-1}$. 
Then the eigenvalues of $N^* g_{S^{n-1}}$ with respect to the metric $g$ are given by $\lambda_1^2,\hdots,\lambda_{n-1}^2$. 
Hence, the eigenvalues of $N^* g_{S^{n-1}} + 4g_{S^1}$ with respect to the metric $\tilde{g}$ are given by $\lambda_1^2,\hdots,\lambda_{n-1}^2,4r^{-2}$. 
Consequently, 
\[
\|d\tilde{N}\|_\tr \leq \sum_{j=1}^{n-1} \lambda_j + 2r^{-1} = \|dN\|_\tr + 2r^{-1}.
\qedhere
\]
\end{proof}

We consider the bundles $\mathcal{S}$ and $\mathcal{E}$ over $M$ defined above. Let $\tilde{\mathcal{S}}$ denote the pull-back of $\mathcal{S}$ under the canonical projection $\tilde{M} \to M$. Similarly, let $\tilde{\mathcal{E}}$ denote the pull-back of $\mathcal{E}$ under the canonical projection $\tilde{M} \to M$. Note that $\tilde{\mathcal{S}}$ and $\tilde{\mathcal{E}}$ are bundles over $\tilde{M}$. Since $n$ is even, the map $\gamma_0 \colon \R^n \to \text{\rm End}(S_0)$ can be extended to a map $\gamma_0 \colon \R^{n+1} \to \text{\rm End}(S_0)$ satisfying the Clifford relations. With this understood, the map $\tilde{N} \colon \tilde{\Sigma} \to S^n$ induces a chirality operator $\tilde{\chi}$. We define $\tilde{\mu}$ and $\tilde{J}$ as in \eqref{def.energy} and \eqref{def.momentum} on $\tilde{M}$ using $\tilde{g}$ and $\tilde{q}$.
Then $\tilde{\mu}$ and $\tilde{J}$ are the pull-backs of $\mu$ and $J$, respectively.

\begin{proposition}
\label{proposition.existence.n.even}
Suppose $M$ has even dimension $n \geq 2$ and the map $N \colon \Sigma \to S^{n-1}$ has positive degree. 
For each $r>0$, we can find a non-trivial section $s \in C^\infty(M,\mathcal{E})$ such that 
\begin{align*} 
\int_M  | \nabla^{\mathcal{E},q} s |^2 
&\leq \tfrac{1}{2} \int_M (-\mu+|J|)  |s|^2 \\ 
&\quad + \tfrac{1}{2} \int_\Sigma (|\tr(q) - q(\nu,\nu)| + \|dN\|_\tr - H) \, |s|^2 + r^{-1} \int_\Sigma |s|^2.  
\end{align*}
Moreover, we can find an element $\hat{t} \in S^1$ with the property that $\gamma(\nu) \circ s = s \circ \gamma_0(\hat{N})$ at each point on $\Sigma$, where $\hat{N} \colon \Sigma \to S^n$ is defined by $\hat{N}(x) = \tilde{N}(x,\hat{t}) = \varphi(N(x),\hat{t})$ for all $x \in \Sigma$. 
\end{proposition}

\begin{proof} 
Since the map $N \colon \Sigma \to S^{n-1}$ has positive degree, it follows that $\tilde{N} \colon \tilde{\Sigma} \to S^n$ has positive degree as well. By Proposition~\ref{proposition.existence.n.odd}, we can find a section $\tilde{s} \in C^\infty(\tilde{M},\tilde{\mathcal{E}})$ such that $\int_{\tilde{M}} |\tilde{s}|^2 > 0$ and 
\begin{align*} 
\int_{\tilde{M}}  | \nabla^{\tilde{\mathcal{E}},\tilde{q}} \tilde{s} |^2 
&\leq 
\tfrac{1}{2} \int_{\tilde{M}} (-\tilde{\mu}+|\tilde{J}|)  |s|^2
+ \tfrac{1}{2} \int_{\tilde{\Sigma}} (|\tr(\tilde{q}) - \tilde{q}(\tilde{\nu},\tilde{\nu})| + \|d\tilde{N}\|_\tr - \tilde{H}) \, |\tilde{s}|^2. 
\end{align*}
Moreover, $\tilde{s}$ satisfies $\tilde{\chi}\tilde{s}=\tilde{s}$ along the boundary $\tilde{\Sigma}$. 
Consequently, there exists an element $\hat{t} \in S^1$ such that $\int_{M \times \{\hat{t}\}} |\tilde{s}|^2 > 0$ and  
\begin{align*} 
\int_{M \times \{\hat{t}\}}  | \nabla^{\tilde{\mathcal{E}},\tilde{q}} \tilde{s} |^2 
&\leq 
\tfrac{1}{2} \int_{M \times \{\hat{t}\}} (-\tilde{\mu}+|\tilde{J}|)  |s|^2 \\ 
&\quad + \tfrac{1}{2} \int_{\Sigma \times \{\hat{t}\}} (|\tr(\tilde{q}) - \tilde{q}(\tilde{\nu},\tilde{\nu})| + \|d\tilde{N}\|_\tr - \tilde{H}) \, |\tilde{s}|^2. 
\end{align*}
It follows from Lemma~\ref{lemma.trace.norm} that $\|d\tilde{N}\|_\tr \leq \|dN\|_\tr + 2r^{-1}$. 
Restricting $\tilde{s}$ to $M \times \{\hat{t}\}$ yields a section $s \in C^\infty(M,\mathcal{E})$ such that $\int_M |s|^2 > 0$ and 
\begin{align*} 
\int_M  | \nabla^{\mathcal{E},q} s |^2 
&\leq \tfrac{1}{2} \int_M (-\mu+|J|)  |s|^2 \\ 
&\quad + \tfrac{1}{2} \int_\Sigma (|\tr(q) - q(\nu,\nu)| + \|dN\|_\tr - H) \, |s|^2 + r^{-1} \int_\Sigma |s|^2.  
\end{align*}
Moreover, since $\tilde{\chi}\tilde{s}=\tilde{s}$ along the boundary $\tilde{\Sigma}$, we conclude that $\gamma(\nu) \circ s = s \circ \gamma_0(\hat{N})$ at each point on $\Sigma$, where the map $\hat{N} \colon \Sigma \to S^n$ is defined by $\hat{N}(x) = \tilde{N}(x,\hat{t}) = \varphi(N(x),\hat{t})$ for all $x \in \Sigma$.
\end{proof}

\begin{corollary}
\label{corollary.existence.n.even}
Suppose that $M$ has even dimension $n \geq 2$ and the map $N \colon \Sigma \to S^{n-1}$ has positive degree. We can find a non-trivial section $s \in H^1(M,\mathcal{E})$ such that 
\begin{align*} 
\int_M  | \nabla^{\mathcal{E},q} s |^2 
&\leq \tfrac{1}{2} \int_M (-\mu+|J|)  |s|^2 
+ \tfrac{1}{2} \int_\Sigma (|\tr(q) - q(\nu,\nu)| + \|dN\|_\tr - H) \, |s|^2.  
\end{align*}
Moreover, we can find an element $\hat{t} \in S^1$ with the property that $\gamma(\nu) \circ s = s \circ \gamma_0(\hat{N})$ at each point on $\Sigma$, where $\hat{N} \colon \Sigma \to S^n$ is defined by $\hat{N}(x) = \tilde{N}(x,\hat{t}) = \varphi(N(x),\hat{t})$ for all $x \in \Sigma$. 
\end{corollary} 

\begin{proof} 
Consider an arbitrary sequence $r_l \to \infty$. 
For each $l$, Proposition~\ref{proposition.existence.n.even} implies the existence of a section $s^{(l)} \in C^\infty(M,\mathcal{E})$ such that $\int_M |s^{(l)}|^2 = 1$ and 
\begin{align*} 
\int_M  | \nabla^{\mathcal{E},q} s^{(l)}|^2 
&\leq \tfrac{1}{2} \int_M (-\mu+|J|)  |s^{(l)}|^2 \\ 
&\quad + \tfrac{1}{2} \int_\Sigma (|\tr(q) - q(\nu,\nu)| + \|dN\|_\tr - H) \, |s^{(l)}|^2 + r_l^{-1} \int_\Sigma |s^{(l)}|^2.  
\end{align*}
Moreover, for each $l$, we can find an element $\hat{t}_l \in S^1$ such that $\gamma(\nu) \circ s^{(l)} = s^{(l)} \circ \gamma_0(\hat{N}_l)$ along $\Sigma$, where the map $\hat{N}_l \colon \Sigma \to S^n$ is defined by $\hat{N}_l(x) = \tilde{N}(x,\hat{t}_l) = \varphi(N(x),\hat{t}_l)$ for all $x \in \Sigma$. This implies 
\[\int_M |\nabla^{\mathcal{E}} s^{(l)}|^2 \leq C \int_M |s^{(l)}|^2 + C \int_\Sigma |s^{(l)}|^2,\] 
where $C$ is a constant that does not depend on $l$. Using a standard interpolation inequality, we can bound 
\[\int_\Sigma |s^{(l)}|^2 \leq C \, \bigg ( \int_M |s^{(l)}|^2 \bigg )^{\frac{1}{2}} \, \bigg ( \int_M |\nabla^{\mathcal{E}} s^{(l)}|^2 \bigg )^{\frac{1}{2}} + C \int_M |s^{(l)}|^2,\] 
where $C$ is a constant that does not depend on $l$. Putting these facts together, we obtain 
\[\int_M |\nabla^{\mathcal{E}} s^{(l)}|^2 \leq C \, \bigg ( \int_M |s^{(l)}|^2 \bigg )^{\frac{1}{2}} \, \bigg ( \int_M |\nabla^{\mathcal{E}} s^{(l)}|^2 \bigg )^{\frac{1}{2}} + C \int_M |s^{(l)}|^2,\] 
where $C$ is a constant that does not depend on $l$. Using Young's inequality, we conclude that 
\[\int_M |\nabla^{\mathcal{E}} s^{(l)}|^2 \leq C \int_M |s^{(l)}|^2,\] 
where $C$ is a constant that does not depend on $l$. Since $\int_M |s^{(l)}|^2 = 1$ for each $l$, it follows that the sequence $s^{(l)}$ is uniformly bounded in $H^1(M,\mathcal{E})$. 

After passing to a subsequence if necessary, we may assume that $\hat{t}_l \to \hat{t}$ for some $\hat{t} \in S^1$. Moreover, we may assume that the sequence $s^{(l)}$ converges weakly in $H^1(M,\mathcal{E})$ to some limit $s \in H^1(M,\mathcal{E})$. 
Then $s_l \to s$ strongly in $L^2(M,\mathcal{E})$. 
Moreover, $s_l|_\Sigma \to s|_\Sigma$ strongly in $L^2(\Sigma,\mathcal{E}|_\Sigma)$. 
Thus, we conclude that $\int_M |s|^2 = 1$ and 
\begin{align*} 
\int_M  | \nabla^{\mathcal{E},q} s |^2 
&\leq \tfrac{1}{2} \int_M (-\mu+|J|)  |s|^2 \\ 
&\quad + \tfrac{1}{2} \int_\Sigma (|\tr(q) - q(\nu,\nu)| + \|dN\|_\tr - H) \, |s|^2.  
\end{align*}
Finally, on the boundary $\Sigma$, we have $\gamma(\nu) \circ s^{(l)} = s^{(l)} \circ \gamma_0(\hat{N}_l)$ for each $l$. Passing to the limit as $l \to \infty$, we conclude that $\gamma(\nu) \circ s = s \circ \gamma_0(\hat{N})$, where the map $\hat{N} \colon \Sigma \to S^n$ is defined by $\hat{N}(x) = \tilde{N}(x,\hat{t}) = \varphi(N(x),\hat{t})$ for all $x \in \Sigma$.
\end{proof}

\section{Proof of Theorem~\ref{rigidity.of.smooth.domains}} \label{section:proofA}

Throughout this section, we consider an initial data set $(M,g,q)$ together with a map $N \colon \Sigma \to S^{n-1}$ satisfying the assumptions in Theorem~\ref{rigidity.of.smooth.domains}.  
In particular, we assume that $\mu \geq |J|$ at each point in $M$ and $H \geq \|dN\|_\tr + |\tr(q) - q(\nu,\nu)|$ at each point on the boundary $\Sigma$. 

We denote the second fundamental form of $\Sigma$ in $M$ (viewed as a (1,1)-tensor) by $h$.
For each $X\in T\Sigma$, we denote by $q(X)^{\text{\rm tan}}$ the part of $q(X)\in TM$ tangential to $\Sigma$.

We begin by proving several preliminary statements.

\begin{proposition} \label{prop:prelim} The following statements hold.
\begin{enumerate}[(a)]
\item There exists a $\nabla^{\mathcal{E},q}$-parallel smooth section $s \in C^{\infty}(M,\mathcal{E})$ satisfying $\chi s = s$ along $\Sigma$.
Considered as a homomorphism field, the section $s$ is invertible at each point. 
\item \label{semidef} At each point $p \in \Sigma$, the eigenvalues of the second fundamental form $h$ agree with the singular values of $dN_p \colon T_p \Sigma \to T_{N(p)} S^{n-1}$. In particular, the eigenvalues of $h$ are nonnegative.
\item We have $q^{\rm tan} = 0$ at each point on $\Sigma$.
\end{enumerate}
\end{proposition}

We prove Proposition~\ref{prop:prelim} first in odd and then in even dimensions.

\subsection{Proof of Proposition~\ref{prop:prelim} in odd dimensions} 
Suppose that $n \geq 3$ is an odd integer. 
By assumption, we have $\mu \geq |J|$ on $M$ and $H \geq \|dN\|_\tr + |\tr(q) - q(\nu,\nu)|$ along $\Sigma$. Therefore, Proposition~\ref{proposition.existence.n.odd} yields non-trivial $\nabla^{\mathcal{E},q}$-parallel section $s \in C^\infty(M,\mathcal{E})$ such that $\chi s = s$ along $\Sigma$. 

\begin{lemma} \label{lemma.boundary.behavior.n.odd}
At each point $p \in \Sigma$, the second fundamental form $h$ and the tensor field $q$ satisfy the relation 
\[\gamma(h(X)) \circ s + i \, \gamma(q(X)^{\text{\rm tan}}) \circ \gamma(\nu) \circ s = s \circ \gamma_0(dN(X))\]
for all $X \in T_p\Sigma$.
\end{lemma}

\begin{proof}
Since $s$ is $\nabla^{\mathcal{E},q}$-parallel, we obtain 
\begin{equation} \label{cov.deriv.s} 
     \nabla_X^{\mathcal{E}} s = - \tfrac{i}{2} Qs(X) = - \tfrac{i}{2} \gamma(q(X)) \circ s
\end{equation}
for every vector field $X$ on $M$. Using the boundary condition $\chi s = s$, we obtain 
\begin{equation} \label{bdry.cond.n.odd} 
   \gamma(\nu) \circ s  =  s \circ \gamma_0(N)  
\end{equation} 
on $\Sigma$. In the next step, we fix a tangential vector field $X$ along $\Sigma$ and differentiate \eqref{bdry.cond.n.odd} along $X$. Using \eqref{cov.deriv.s} we find for the left hand side 
\begin{align*}
\nabla_X^{\mathcal{E}}(\gamma(\nu) \circ s)
&=
\gamma(D_X\nu) \circ s + \gamma(\nu) \circ \nabla_X^{\mathcal{E}} s \\
&=
\gamma(h(X)) \circ s - \tfrac{i}{2} \gamma(\nu) \circ \gamma(q(X)) \circ s
\end{align*}
and for the right hand side
\begin{align*}
\nabla_X^{\mathcal{E}}(s \circ \gamma_0(N))
&=
(\nabla_X^{\mathcal{E}} s) \circ \gamma_0(N) + s \circ \gamma_0(dN(X)) \\
&=
-\tfrac{i}{2} \gamma(q(X)) \circ s\circ \gamma_0(N) + s \circ \gamma_0(dN(X)) .
\end{align*}
Thus \eqref{bdry.cond.n.odd} implies
\begin{align*} 
\gamma(h(X)) & \circ s - s \circ \gamma_0(dN(X))  \\ 
&= 
- \tfrac{i}{2} \big(\gamma(q(X)) \circ s\circ \gamma_0(N)   - \gamma(\nu) \circ \gamma(q(X)) \circ s \big) \\ 
&= 
- \tfrac{i}{2} \big( \gamma(q(X)) \circ \gamma(\nu) \circ s - \gamma(\nu) \circ \gamma(q(X)) \circ s \big) \\ 
&= 
- i \, \gamma(q(X)^{\text{\rm tan}}) \circ \gamma(\nu) \circ s . 
\end{align*}
To justify the last equality, we decompose $q(X)$ into its tangential and normal components and observe that $\gamma(q(X)^{\text{\rm tan}})$ anticommutes with $\gamma(\nu)$.
\end{proof}

\begin{lemma}\label{lemma.parallel.spinors.n.odd}
Let $\sigma \in S_0$. Then $s\sigma$ is a $\nabla^{\mathcal{S},q}$-parallel section of $\mathcal{S}$. In particular, if $s\sigma$ vanishes at some point in $M$, then $s\sigma$ vanishes at each point in $M$.
\end{lemma}

\begin{proof}
This follows directly from the fact that $\nabla^{\mathcal{E},q}$ is the tensor product connection of $\nabla^{\mathcal{S},q}$ and the flat connection on the trivial bundle over $M$ with fiber $S_0^*$.
\end{proof}

We define a linear subspace $L \subset S_0$ by $L := \{\sigma \in S_0 \colon \text{\rm $s\sigma=0$ at each point in $M$}\}$.

\begin{lemma} \label{lemma.L.invariant.n.odd} 
The subspace $L\subset  S_0 $ is invariant under $\gamma_0(N(p))$ for each point $p \in \Sigma$.
\end{lemma}

\begin{proof}
Let us fix a point $p \in \Sigma$ and an element $\sigma \in L$. Using the boundary condition $\chi s = s$, we obtain 
\[\gamma(\nu) \circ s = s \circ \gamma_0(N)\] 
at the point $p$. Since $\sigma \in L$, it follows that 
\[0 = \gamma(\nu) s \sigma = s \gamma_0(N) \sigma\] 
at the point $p$. 
Using Lemma~\ref{lemma.parallel.spinors.n.odd}, we conclude that $\gamma_0(N(p)) \sigma \in L$. 
\end{proof}

\begin{lemma}\label{lemma.maxrank.n.odd}
We have $L=\{0\}$. 
Moreover, the section $s$, considered as a homomorphism field, is invertible at each point in $M$.
\end{lemma}

\begin{proof}
It follows from Lemma~\ref{lemma.L.invariant.n.odd} that $L$ is invariant under $\gamma_0(N(p))$ for each point $p \in \Sigma$. Since the map $N \colon \Sigma \to S^{n-1}$ has non-zero degree, we know that the image of $N$ is all of $S^{n-1}$. Consequently, $L$ is invariant under $\gamma_0(\xi)$ for each $\xi \in \R^n$. 

The spinor representation of the Clifford algebra $\C \mathrm{l}(\R^n)\to \mathrm{End}( S_0 )$ is surjective. Therefore, the elements $\gamma_0(\xi)$, $\xi \in \R^n$, generate all of $\mathrm{End}(S_0)$. Putting these facts together, we conclude that $L$ is invariant under $\mathrm{End}(S_0)$. Therefore, $L=\{0\}$ or $L= S_0 $. The latter case is impossible as this would imply that $s$ vanishes identically. To summarize, we have shown that $L=\{0\}$. Using Lemma~\ref{lemma.parallel.spinors.n.odd}, we conclude that $s$ is injective at each point in $M$. For dimensional reasons, $s$ is invertible at each point in $M$.
\end{proof}

\begin{lemma}
\label{lemma.boundary.rigidity.odd.dim}
Let $p$ be a point on the boundary $\Sigma$. 
Let $e_1,\dots,e_{n-1}$ be an orthonormal basis of $T_p\Sigma$ and $e^0_1,\dots,e^0_{n-1}$ an orthonormal basis of $T_{N(p)} S^{n-1}$ such that $dN(e_j)=\lambda_j e^0_j$, where $\lambda_j\geq 0$ denote the singular values of $dN$. Then $h(e_j) = \lambda_j e_j$ and $q(e_j)^{\text{\rm tan}} = 0$ for $j=1,\hdots,n-1$.
\end{lemma} 

\begin{proof}
Applying Lemma~\ref{lemma.boundary.behavior.n.odd} with $X=e_j$, we obtain 
\begin{equation} 
\label{key.identity}
\gamma(h(e_j)) \circ s + i \, \gamma(q(e_j)^{\text{\rm tan}}) \circ \gamma(\nu) \circ s = \lambda_j \, s \circ \gamma_0(e^0_j) 
\end{equation}
for $j=1,\hdots,n-1$. This implies 
\begin{align*}
H \, &|s|^2 + i \, (\text{\rm tr}(q) - q(\nu,\nu)) \, \langle \gamma(\nu) \circ s,s \rangle \\ 
&= -\sum_j \langle \gamma(e_j) \circ \gamma(h(e_j)) \circ s,s \rangle - i \sum_j \, \langle \gamma(e_j) \circ \gamma(q(e_j)^{\text{\rm tan}}) \circ \gamma(\nu) \circ s,s \rangle \\ 
&= -\sum_j \lambda_j \, \langle \gamma(e_j) \circ s \circ \gamma_0(e^0_j),s \rangle. 
\end{align*}
In the next step, we take the absolute value on both sides. Using the triangle inequality, we obtain 
\begin{equation} 
\label{consequence.of.key.identity}
H \, |s|^2 - |\text{\rm tr}(q) - q(\nu,\nu)| \, |s|^2 \leq \|dN\|_\tr \, |s|^2. 
\end{equation}
On the other hand, our assumption implies that $\|dN\|_\tr \leq H - |\text{\rm tr}(q) - q(\nu,\nu)|$. 
Consequently, equality holds in \eqref{consequence.of.key.identity}. 
From this, we deduce that $-\gamma(e_j) \circ s \circ \gamma_0(e^0_j) = s$ whenever $\lambda_j > 0$. 
Substituting this into \eqref{key.identity} gives 
\[\gamma(h(e_j)) \circ s + i \, \gamma(q(e_j)^{\text{\rm tan}}) \circ \gamma(\nu) \circ s = \lambda_j \, s \circ \gamma_0(e^0_j) = \lambda_j \, \gamma(e_j) \circ s\] 
for $j=1,\hdots,n-1$. 
Since $s \in \text{\rm Hom}(S_0,\mathcal{S}|_p)$ is invertible by Lemma~\ref{lemma.maxrank.n.odd}, it follows that 
\[\gamma(h(e_j)) + i \, \gamma(q(e_j)^{\text{\rm tan}}) \circ \gamma(\nu) = \lambda_j \, \gamma(e_j) \in \text{\rm End}(\mathcal{S}|_p)\] 
for $j=1,\hdots,n-1$. 
Note that 
\[
i \, \gamma(q(e_j)^{\text{\rm tan}}) \circ \gamma(\nu) \in \text{\rm End}(\mathcal{S}|_p)
\] 
is self-adjoint, while $\gamma(h(e_j)) \in \text{\rm End}(\mathcal{S}|_p)$ and $\gamma(e_j) \in \text{\rm End}(\mathcal{S}|_p)$ are skew-adjoint. Thus, we conclude that $\gamma(h(e_j)) = \lambda_j \, \gamma(e_j) \in \text{\rm End}(\mathcal{S}|_p)$ and $i \, \gamma(q(e_j)^{\text{\rm tan}}) \circ \gamma(\nu) = 0 \in \text{\rm End}(\mathcal{S}|_p)$ for $j=1,\hdots,n-1$. From this, the assertion follows.
\end{proof}

In summary, we proved all the assertions in Proposition~\ref{prop:prelim} in odd dimensions.

\subsection{Proof of Proposition~\ref{prop:prelim} in even dimensions} \label{subsec:even} 
We now turn to the even-dimensional case of Proposition~\ref{prop:prelim}. 
Suppose that $n \geq 2$ is an even integer. 
Since $\mu \geq |J|$ on $M$ and $H \geq \|dN\|_\tr + |\tr(q) - q(\nu,\nu)|$ along $\Sigma$, Corollary~\ref{corollary.existence.n.even} yields a non-trivial $\nabla^{\mathcal{E},q}$-parallel section  $s \in H^1(M,\mathcal{E})$. 
Clearly, $s$ is smooth. 
Moreover, we can find an element $\hat{t} \in S^1$ such that $\gamma(\nu) \circ s = s \circ \gamma_0(\hat{N})$ along $\Sigma$. 
Here, the map $\hat{N} \colon \Sigma \to S^n$ is defined by 
\[
\hat{N}(x) = \tilde{N}(x,\hat{t}) = \varphi(N(x),\hat{t})
\]
for all $x \in \Sigma$. 

\begin{lemma} \label{lemma.boundary.behavior.n.even}
At each point $p \in \Sigma$, the second fundamental form $h$ and the tensor field $q$ satisfy the relation 
\[\gamma(h(X)) \circ s + i \, \gamma(q(X)^{\text{\rm tan}}) \circ \gamma(\nu) \circ s  = s \circ \gamma_0(d\hat{N}(X))\]
for all $X\in T_p\Sigma$.
\end{lemma}

\begin{proof}
Analogous to the proof of Lemma~\ref{lemma.boundary.behavior.n.odd}. 
\end{proof}

\begin{lemma} 
\label{trace.norm.of.N.and.hat.N}
We have $\|d\hat{N}\|_\tr = \|dN\|_\tr$ at each point on $\Sigma$.
\end{lemma}

\begin{proof} 
Let us fix a point $p \in \Sigma$. Let $e_1,\dots,e_{n-1}$ be an orthonormal basis of $T_p\Sigma$. 
Using Lemma~\ref{lemma.boundary.behavior.n.even}, we obtain 
\begin{align*}
H \, &|s|^2 + i \, (\text{\rm tr}(q) - q(\nu,\nu)) \, \langle \gamma(\nu) \circ s,s \rangle \\ 
&= -\sum_j \langle \gamma(e_j) \circ \gamma(h(e_j)) \circ s,s \rangle - i \sum_j \, \langle \gamma(e_j) \circ \gamma(q(e_j)^{\text{\rm tan}}) \circ \gamma(\nu) \circ s,s \rangle \\ 
&= -\sum_j \langle \gamma(e_j) \circ s \circ \gamma_0(d\hat{N}(e_j)),s \rangle. 
\end{align*}
In the next step, we take the absolute value on both sides. This gives 
\[H \, |s|^2 - |\text{\rm tr}(q) - q(\nu,\nu)| \, |s|^2 \leq \|d\hat{N}\|_\tr \, |s|^2.\] 
Since $|s|^2 > 0$ at each point in $M$, it follows that $H - |\text{\rm tr}(q) - q(\nu,\nu)| \leq \|d\hat{N}\|_\tr$. On the other hand, $\|d\hat{N}\|_\tr \leq \|dN\|_\tr$ by definition of $\hat{N}$. 
Finally, our assumption implies that 
\[
\|dN\|_\tr \leq H - |\text{\rm tr}(q) - q(\nu,\nu)|.
\]
Putting these facts together, the assertion follows.
\end{proof}

\begin{lemma} \label{lemma.hat.t}
The map $\varphi(\cdot,\hat{t}) \colon S^{n-1} \to S^n$ is the standard inclusion $S^{n-1} \hookrightarrow S^n$.
In particular, the map $\hat{N} \colon \Sigma \to S^n$ is the composition of the map $N \colon \Sigma \to S^{n-1}$ with the standard inclusion $S^{n-1} \hookrightarrow S^n$.
\end{lemma}

\begin{proof} 
Suppose that the assertion is false. By Lemma~\ref{lemma.existence.of.varphi}, the map $\varphi(\cdot,\hat{t})$ is Lipschitz continuous with a Lipschitz constant strictly less than $1$. Therefore, at each point on $\Sigma$, we either have $\|d\hat{N}\|_\tr < \|dN\|_\tr$ or $\|dN\|_\tr = 0$. Using Lemma~\ref{trace.norm.of.N.and.hat.N}, we conclude that $\|dN\|_\tr = 0$ at each point on $\Sigma$. This implies that $N$ is locally constant. This contradicts the assumption that $N$ has non-zero degree.
\end{proof}

\begin{lemma}\label{lemma.parallel.spinors.n.even}
Let $\sigma \in S_0$. Then $s\sigma$ is a $\nabla^{\mathcal{S},q}$-parallel section of $\mathcal{S}$. In particular, if $s\sigma$ vanishes at some point in $M$, then $s\sigma$ vanishes at each point in $M$.
\end{lemma}

\begin{proof} 
Analogous to the proof of Lemma~\ref{lemma.parallel.spinors.n.odd}. 
\end{proof}

As above, we define a linear subspace $L \subset S_0$ by  
\[L := \{\sigma \in S_0 \colon \text{\rm $s\sigma=0$ at each point in $M$}\}.\] 

\begin{lemma} \label{lemma.L.invariant.n.even} 
The subspace $L\subset  S_0 $ is invariant under $\gamma_0(\hat{N}(p))$ for each point $p \in \Sigma$.
\end{lemma}

\begin{proof}
Analogous to the proof of Lemma~\ref{lemma.L.invariant.n.odd}.
\end{proof}

\begin{lemma}\label{lemma.maxrank.n.even}
We have $L=\{0\}$. Moreover, the section $s$, considered as a homomorphism field, is invertible on each point in $M$.
\end{lemma}

\begin{proof}
Since the map $N \colon \Sigma \to S^{n-1}$ has non-zero degree, we know that the image of $N \colon \Sigma \to S^{n-1}$ is all of $S^{n-1}$. 
Using Lemma~\ref{lemma.hat.t}, we conclude that the image of the map $\hat{N} \colon \Sigma \to S^n$ is the equator in $S^n$. In particular, $\text{\rm span}\{\hat{N}(p) \colon p \in \Sigma\} = \R^n \subset \R^{n+1}$. 
By Lemma~\ref{lemma.L.invariant.n.even}, $L$ is invariant under $\gamma_0(\xi)$ for all $\xi \in \R^n$. 
Since the volume element acts as a scalar multiple of the identity, we deduce that $L$ is invariant under $\gamma_0(\xi)$ for all $\xi \in \R^{n+1}$. 
Arguing as in the proof of Lemma~\ref{lemma.maxrank.n.odd}, the assertion follows. 
\end{proof}

\begin{lemma}
\label{lemma.boundary.rigidity.even.dim}
Let $p$ be a point on the boundary $\Sigma$. Let $e_1,\dots,e_{n-1}$ be an orthonormal basis of $T_p\Sigma$ and let $e^0_1,\dots,e^0_{n-1}$ be a set of orthonormal vectors in $T_{\hat{N}(p)} S^n$ such that $d\hat{N}(e_j)=\lambda_j e^0_j$, where $\lambda_j\geq 0$ denote the singular values of $d\hat{N}$. Then $h(e_j) = \lambda_j e_j$ and $q(e_j)^{\text{\rm tan}} = 0$ for $j=1,\hdots,n-1$.
\end{lemma} 

\begin{proof}
The proof is analogous to the proof of Lemma~\ref{lemma.boundary.rigidity.odd.dim}.
\end{proof}

In summary, we proved all the assertions in Proposition~\ref{prop:prelim} in even dimensions.

\subsection{Existence of a local isometric immersion into $\R^{n,1}$}

\label{immersion.into.Minkowski.space}

In this subsection, we construct a local isometric immersion of $(M,g)$ into the Minkowski spacetime $\R^{n,1}$.

\begin{definition}
We define a vector bundle $\mathcal{T}$ over $M$ by $\mathcal{T} = TM \oplus \R$. We define a bundle metric of signature $(n,1)$ on $\mathcal{T}$ by 
\[\langle (Y,\varphi),(Z,\psi) \rangle_{\mathcal{T}} = g(Y,Z) - \varphi \psi\] 
for $(Y,\varphi),(Z,\psi) \in \mathcal{T}$. We define a connection $\nabla^{\mathcal{T},q}$ on $\mathcal{T}$ by 
\[\nabla_X^{\mathcal{T},q} (Y,\varphi) = (D_X Y - \varphi \, q(X),X(\varphi) - q(X,Y))\] 
for every vector field $X$ and every section $(Y,\varphi)$ of $\mathcal{T}$. 
\end{definition}

\begin{lemma}
\label{metric.compatibility}
The connection $\nabla^{\mathcal{T},q}$ is compatible with the bundle metric $\langle \cdot,\cdot \rangle_{\mathcal{T}}$.
\end{lemma}

\begin{proof} 
Let $X$ be a vector field on $M$, and let $(Y,\varphi)$ and $(Z,\psi)$ be two sections of $\mathcal{T}$. Then 
\[\langle \nabla_X^{\mathcal{T},q} (Y,\varphi),(Z,\psi) \rangle_{\mathcal{T}} = g(D_X Y,Z) - \varphi \, q(X,Z) - X(\varphi) \, \psi + q(X,Y) \, \psi\] 
and 
\[\langle (Y,\varphi),\nabla_X^{\mathcal{T},q} (Z,\psi) \rangle_{\mathcal{T}} = g(Y,D_X Z) - q(X,Y) \, \psi - \varphi \, X(\psi) + \varphi \, q(X,Z).\]
This gives 
\begin{align*} 
\langle \nabla_X^{\mathcal{T},q} (Y,\varphi),(Z,\psi) \rangle_{\mathcal{T}}& + \langle (Y,\varphi),\nabla_X^{\mathcal{T},q} (Z,\psi) \rangle_{\mathcal{T}} \\ 
&= X(g(Y,Z) - \varphi \psi) = X(\langle (Y,\varphi),(Z,\psi) \rangle_{\mathcal{T}}). 
\end{align*} 
This completes the proof of Lemma~\ref{metric.compatibility}. 
\end{proof}

\begin{lemma}
\label{parallel.sections}
Suppose that $u,v$ are $\nabla^{\mathcal{S},q}$-parallel sections of $\mathcal{S}$. 
We define a vector field $Y$ on $M$ so that $g(Y,Z) = \text{\rm Re}(i \, \langle \gamma(Z) \, u,v \rangle)$ for all $Z\in TM$. 
Moreover, we define a scalar function $\varphi$ on $M$ by $\varphi = -\text{\rm Re}(\langle u,v \rangle)$. Then $(Y,\varphi)$ is a $\nabla^{\mathcal{T},q}$-parallel section of $\mathcal{T}$.
\end{lemma}

\begin{proof} 
Let $X$ be a vector field on $M$. By assumption, 
\[\nabla_X^{\mathcal{S}} u + \tfrac{i}{2} \, \gamma(q(X)) \, u = \nabla_X^{\mathcal{S}} v + \tfrac{i}{2} \, \gamma(q(X)) \, v = 0.\] 
This implies 
\begin{align*} 
g(D_X Y,Z)
&= \text{\rm Re}(i \, \langle \gamma(Z) \, \nabla_X^{\mathcal{S}} u,v \rangle) + \text{\rm Re}(i \, \langle \gamma(Z) \, u,\nabla_X^{\mathcal{S}} v \rangle) \\ 
&= \tfrac{1}{2} \, \text{\rm Re}(\langle \gamma(Z) \, \gamma(q(X)) \, u,v \rangle) - \tfrac{1}{2} \, \text{\rm Re}(\langle \gamma(Z) \, u,\gamma(q(X)) \, v \rangle) \\ 
&= \tfrac{1}{2} \, \text{\rm Re}(\langle \gamma(Z) \, \gamma(q(X)) \, u,v \rangle) + \tfrac{1}{2} \, \text{\rm Re}(\langle \gamma(q(X)) \, \gamma(Z) \, u,v \rangle) \\ 
&= -q(X,Z) \, \text{\rm Re}(\langle u,v \rangle) \\ 
&= q(X,Z) \, \varphi
\end{align*} 
and 
\begin{align*} 
X(\varphi) 
&= -\text{\rm Re}(\langle \nabla_X^{\mathcal{S}} u,v \rangle) - \text{\rm Re}(\langle u,\nabla_X^{\mathcal{S}} v \rangle) \\ 
&= \tfrac{1}{2} \, \text{\rm Re}(i \, \langle \gamma(q(X)) \, u,v \rangle) - \tfrac{1}{2} \, \text{\rm Re}(i \, \langle u,\gamma(q(X)) \, v \rangle) \\ 
&= \text{\rm Re}(i \, \langle \gamma(q(X)) \, u,v \rangle) \\ 
&= q(X,Y). 
\end{align*} 
This completes the proof of Lemma~\ref{parallel.sections}. 
\end{proof}

\begin{lemma}
\label{span}
Suppose that $p$ is an arbitrary point in $M$. Consider the set of all pairs $(Y,\varphi) \in \mathcal{T}|_p$ with the property that there exist $u,v \in \mathcal{S}|_p$ such that $g(Y,Z) = \text{\rm Re}(i \, \langle \gamma(Z) \, u,v \rangle)$ for all $Z \in TM$ and $\varphi = -\text{\rm Re}(\langle u,v \rangle)$. This set spans $\mathcal{T}|_p$. 
\end{lemma} 

\begin{proof} 
Suppose that the assertion is false. 
Then there exists a non-zero pair $(Z,\psi) \in \mathcal{T}|_p$ such that 
\[
\text{\rm Re}(i \, \langle \gamma(Z) \, u,v \rangle) + \text{\rm Re}(\psi \, \langle u,v \rangle) = 0
\] 
for all $u,v \in \mathcal{S}|_p$. 
This implies that $i \, \gamma(Z) + \psi \, \id = 0 \in \text{\rm End}(\mathcal{S}|_p)$. 
From this, we deduce that $Z=0$ and $\psi=0$. 
This is a contradiction. 
\end{proof}

\begin{proposition}
\label{extended.tangent.bundle.trivial}
The bundle $\mathcal{T}$ admits a trivialization by $\nabla^{\mathcal{T},q}$-parallel sections.
\end{proposition} 

\begin{proof} 
According to Proposition~\ref{prop:prelim}, there exists a $\nabla^{\mathcal{E},q}$-parallel section $s$ of the bundle $\text{\rm Hom}(S_0,\mathcal{S})$ with the property that $s$ is invertible at each point in $M$. Consequently, the spinor bundle $\mathcal{S}$ admits a trivialization by $\nabla^{\mathcal{S},q}$-parallel sections. Using Lemma~\ref{parallel.sections}, we obtain a collection of $\nabla^{\mathcal{T},q}$-parallel sections of $\mathcal{T}$. At each point $p \in M$, these $\nabla^{\mathcal{T},q}$-parallel sections span the fiber $\mathcal{T}|_p$ (see Lemma~\ref{span}). This completes the proof of Proposition~\ref{extended.tangent.bundle.trivial}. 
\end{proof}

Let $(E^{(0)},\zeta^{(0)}),(E^{(1)},\zeta^{(1)}),\hdots,(E^{(n)},\zeta^{(n)})$ be a collection of $\nabla^{\mathcal{T},q}$-parallel sections that trivialize $\mathcal{T}$. In view of Lemma~\ref{metric.compatibility}, we may write 
\[\langle (E^{(\alpha)},\zeta^{(\alpha)}),(E^{(\beta)},\zeta^{(\beta)}) \rangle_{\mathcal{T}} = m_{\alpha\beta},\] 
where $m_{\alpha\beta}$ is a constant $(n+1) \times (n+1)$-matrix. The matrix $m_{\alpha\beta}$ is invertible and has signature $(n,1)$. By taking suitable linear combinations, we may arrange that 
\[m_{\alpha\beta} = \begin{cases} -1 & \text{\rm for $\alpha=\beta=0$} \\ 1 & \text{\rm for $\alpha=\beta>0$} \\ 0 & \text{\rm for $\alpha \neq \beta$} \end{cases}\] 
for $\alpha,\beta = 0,1,\hdots,n$. Then 
\[\langle (Y,\varphi),(Z,\psi) \rangle_{\mathcal{T}} = \sum_{\alpha,\beta=0}^n m_{\alpha\beta} \, \langle (E^{(\alpha)},\zeta^{(\alpha)}),(Y,\varphi) \rangle_{\mathcal{T}} \, \langle (E^{(\beta)},\zeta^{(\beta)}),(Z,\psi) \rangle_{\mathcal{T}}\] 
for all sections $(Y,\varphi)$ and $(Z,\psi)$ of $\mathcal{T}$. 
Choosing $\varphi=\psi=0$, this implies 
\begin{equation} 
\label{tan.tan}
g(Y,Z) = \sum_{\alpha,\beta=0}^n m_{\alpha\beta} \, g(E^{(\alpha)},Y) \, g(E^{(\beta)},Z) 
\end{equation}
for all $Y,Z \in TM$. 
Moreover, choosing $\varphi=0$ and $Z=0$, we find
\begin{equation} 
\label{tan.nor}
0 = \sum_{\alpha,\beta=0}^n m_{\alpha\beta} \, g(E^{(\alpha)},Y) \, \zeta^{(\beta)} = 0 
\end{equation} 
for all $Y \in TM$.
Finally, choosing $Y=Z=0$ and $\varphi=\psi=1$ yields
\begin{equation}
\label{nor.nor}
-1 = \sum_{\alpha,\beta=0}^n m_{\alpha\beta} \, \zeta^{(\alpha)} \, \zeta^{(\beta)}. 
\end{equation} 

Since $(E^{(\alpha)},\zeta^{(\alpha)})$ is $\nabla^{\mathcal{T},q}$-parallel, it follows that the covariant derivative of $E^{(\alpha)}$ is given by $\zeta^{(\alpha)}q$ and is hence symmetric. 
Thus, the one-form $g(E^{(\alpha)},\cdot)$ is closed. We may locally write $g(E^{(\alpha)},\cdot) = df^{(\alpha)}$ for some scalar function $f^{(\alpha)}$. It follows from \eqref{tan.tan} that $g = \sum_{\alpha,\beta=0}^n m_{\alpha\beta} \, df^{(\alpha)} \otimes df^{(\beta)}$. In other words, the map $F = (f^{(0)},f^{(1)},\hdots,f^{(n)})$ is a local isometric immersion of $(M,g)$ into the Minkowski spacetime $\R^{n,1}$. For each vector $X \in TM$, we have 
\begin{align*} 
dF(X) 
&= (df^{(0)}(X),df^{(1)}(X),\hdots,df^{(n)}(X)) \\ 
&= (g(E^{(0)},X),g(E^{(1)},X),\hdots,g(E^{(n)},X)) \in \R^{n,1}. 
\end{align*}
It follows from \eqref{tan.nor} and \eqref{nor.nor} that the unit normal vector to the immersion $F$ in $\R^{n,1}$ is given by $(\zeta^{(0)},\zeta^{(1)},\hdots,\zeta^{(n)})$. Since $(E^{(\alpha)},\zeta^{(\alpha)})$ is $\nabla^{\mathcal{T},q}$-parallel, it is easy to see that the second fundamental form associated with the immersion $F$ is equal to $q$.

If $M$ is simply connected, the isometric immersion $F = (f^{(0)},f^{(1)},\hdots,f^{(n)})$ is globally well-defined. If $M$ is not simply connected, then $F$ may not be globally well-defined, but its differential is globally well-defined.

\subsection{Proof of Theorem~\ref{rigidity.of.smooth.domains} in the simply connected case} 

\label{simply.connected}

In this subsection, we prove Theorem~\ref{rigidity.of.smooth.domains} in the special case that $M$ is simply connected. In this case, there exists an isometric immersion $F$ of $M$  into the Minkowski spacetime $\R^{n,1}$ as a spacelike hypersurface. 
Let $T$ denote the normal vector to $M$ in $\R^{n,1}$, normalized so that $\langle T,T \rangle_{\R^{n,1}} = -1$. The second fundamental form is equal to $q$, so that $q(X,Y) = \langle \bar{D}_X T,Y \rangle_{\R^{n,1}}$ for $X,Y \in TM$. Let $\mathcal{S}$ denote the spinor bundle of $M$, and let $\nabla^{\mathcal{S}}$ denote the usual connection on $\mathcal{S}$.

\begin{definition} 
We define a vector bundle $\mathcal{W}$ over $M$ by $\mathcal{W} = \mathcal{S} \oplus \mathcal{S}$. We define a connection $\nabla^{\mathcal{W}}$ on $\mathcal{W}$ by 
\[\nabla_X^{\mathcal{W}} (u,v) = (\nabla_X^{\mathcal{S},q} u,\nabla_X^{\mathcal{S},-q} v),\] 
where $X$ denotes a vector field on $M$, and $u,v$ are sections of $\mathcal{S}$.
\end{definition}

\begin{definition} 
For each point $p \in M$, we define a map $\gamma^{\mathcal{W}} \colon \R^{n,1} \to \text{\rm End}(\mathcal{W}|_p)$ by 
\[\gamma^{\mathcal{W}}(T) \, (u,v) = (-iv,iu)\] 
and
\[\gamma^{\mathcal{W}}(X) \, (u,v) = (\gamma(X) \, v,\gamma(X) \, u),\] 
where $X$ denotes a vector field on $M$ and $u,v \in \mathcal{S}|_p$.
\end{definition}

\begin{lemma} 
\label{lorentzian.clifford.relations}
The map $\gamma^{\mathcal{W}} \colon \R^{n,1} \to \text{\rm End}(\mathcal{W}|_p)$ satisfies the Clifford relations with respect to the Lorentzian inner product on $\R^{n,1}$. 
In other words, 
\[
\gamma^{\mathcal{W}}(a) \, \gamma^{\mathcal{W}}(b) + \gamma^{\mathcal{W}}(b) \, \gamma^{\mathcal{W}}(a) = -2 \, \langle a,b \rangle_{\R^{n,1}} \, \id
\] 
for all $a,b \in \R^{n,1}$.
\end{lemma} 

\begin{proof} 
We compute 
\[
\gamma^{\mathcal{W}}(X) \, \gamma^{\mathcal{W}}(Y) + \gamma^{\mathcal{W}}(Y) \, \gamma^{\mathcal{W}}(X) = -2 \, g(X,Y) \, \id
\] 
for all $X,Y \in TM$. 
Moreover, 
\[
\gamma^{\mathcal{W}}(T) \, \gamma^{\mathcal{W}}(X) + \gamma^{\mathcal{W}}(X) \, \gamma^{\mathcal{W}}(T) = 0\] 
for all $X \in TM$. 
Finally, $(\gamma^{\mathcal{W}}(T))^2 = \id$.
This completes the proof of Lemma~\ref{lorentzian.clifford.relations}.
\end{proof}

\begin{lemma}
\label{compatibility}
Suppose that $a$ is a constant vector in $\R^{n,1}$. Then 
\[\nabla_X^{\mathcal{W}} (\gamma^{\mathcal{W}}(a) \, (u,v)) = \gamma^{\mathcal{W}}(a) \, \nabla_X^{\mathcal{W}} (u,v)\] 
for all $X \in TM$. In particular, $\gamma^{\mathcal{W}}(a)$ maps parallel sections of $\mathcal{W}$ to parallel sections of $\mathcal{W}$.
\end{lemma}

\begin{proof}
Let $X$ be a section of $TM$. We decompose $a = Y+fT$, where $Y$ is a section of $TM$ and $f$ is a scalar function on $M$. Then $\bar{D}_X Y + \bar{D}_X (fT) = 0$, where $\bar{D}$ denotes the standard flat connection on $\R^{n,1}$. Using the identity $\bar{D}_X Y = D_X Y + q(X,Y) \, T$, we obtain 
\begin{align} 
\label{part.1}
\nabla_X^{\mathcal{W}}& (\gamma^{\mathcal{W}}(Y) \, (u,v)) - \gamma^{\mathcal{W}}(Y) \, \nabla_X^{\mathcal{W}}(u,v) - \gamma^{\mathcal{W}}(\bar{D}_X Y) \, (u,v) \notag \\ 
&= \nabla_X^{\mathcal{W}}(\gamma(Y) \, v,\gamma(Y) \, u) - \gamma^{\mathcal{W}}(Y) \, (\nabla_X^{\mathcal{S},q} u,\nabla_X^{\mathcal{S},-q} v) \notag \\ 
&\quad - \gamma^{\mathcal{W}}(D_X Y) \, (u,v) - q(X,Y) \, \gamma^{\mathcal{W}}(T) \, (u,v) \notag \\ 
&= (\nabla_X^{\mathcal{S},q}(\gamma(Y) \, v),\nabla_X^{\mathcal{S},-q}(\gamma(Y) \, u)) - (\gamma(Y) \, \nabla_X^{\mathcal{S},-q} v,\gamma(Y) \, \nabla_X^{\mathcal{S},q} u) \\ 
&\quad - (\gamma(D_X Y) \, v,\gamma(D_X Y) \, u) - q(X,Y) \, (-iv,iu) \notag \\ 
&= 0. \notag
\end{align}
Using the identity $\bar{D}_X T = q(X)$, we obtain 
\begin{align*} 
&\nabla_X^{\mathcal{W}} (\gamma^{\mathcal{W}}(T) \, (u,v)) - \gamma^{\mathcal{W}}(T) \, \nabla_X^{\mathcal{W}}(u,v) - \gamma^{\mathcal{W}}(\bar{D}_X T) \, (u,v) \\ 
&= \nabla_X^{\mathcal{W}}(-iv,iu) - \gamma^{\mathcal{W}}(T) \, (\nabla_X^{\mathcal{S},q} u,\nabla_X^{\mathcal{S},-q} v) - \gamma^{\mathcal{W}}(q(X)) \, (u,v) \\ 
&= (-i \, \nabla_X^{\mathcal{S},q} v,i \, \nabla_X^{\mathcal{S},-q} u) - (-i \, \nabla_X^{\mathcal{S},-q} v,i \, \nabla_X^{\mathcal{S},q} u) - (\gamma(q(X)) \, v,\gamma(q(X)) \, u) \\ 
&= 0. 
\end{align*} 
This implies 
\begin{equation} 
\label{part.2}
\nabla_X^{\mathcal{W}} (\gamma^{\mathcal{W}}(fT) \, (u,v)) - \gamma^{\mathcal{W}}(fT) \, \nabla_X^{\mathcal{W}}(u,v) - \gamma^{\mathcal{W}}(\bar{D}_X (fT)) \, (u,v) = 0. 
\end{equation}
In the next step, we add \eqref{part.1} and \eqref{part.2}. Using the identity $\bar{D}_X Y + \bar{D}_X (fT) = 0$, we conclude that 
\[\nabla_X^{\mathcal{W}} (\gamma^{\mathcal{W}}(a) \, (u,v)) - \gamma^{\mathcal{W}}(a) \, \nabla_X^{\mathcal{W}} (u,v) = 0.\] 
This completes the proof of Lemma~\ref{compatibility}.
\end{proof}

\begin{definition} 
For each point $p \in M$, we define a map $\gamma^{\mathcal{W}} \colon \Lambda^2 \R^{n,1} \to \text{\rm End}(\mathcal{W}|_p)$ by 
\[\gamma^{\mathcal{W}}(a \wedge b) = \tfrac{1}{2} \, (\gamma^{\mathcal{W}}(a) \, \gamma^{\mathcal{W}}(b) - \gamma^{\mathcal{W}}(b) \, \gamma^{\mathcal{W}}(a)).\] 
Note that $\gamma^{\mathcal{W}}(a \wedge b) = \gamma^{\mathcal{W}}(a) \, \gamma^{\mathcal{W}}(b)$ whenever $\langle a,b \rangle_{\R^{n,1}} = 0$.
\end{definition}

By Proposition~\ref{prop:prelim}, there exists a $\nabla^{\mathcal{E},q}$-parallel section $s$ of the bundle $\text{\rm Hom}(S_0,\mathcal{S})$ with the property that $s$ is invertible at each point in $M$. Moreover, $s$ satisfies the boundary condition $\gamma(\nu) \circ s = s \circ \gamma_0(N)$ at each point on $\Sigma$. 

Replacing $q$ by $-q$, we can find a $\nabla^{\mathcal{E},-q}$-parallel section $t$ of the bundle $\text{\rm Hom}(S_0,\mathcal{S})$ with the property that $t$ is invertible at each point in $M$. Moreover, $t$ satisfies the boundary condition $\gamma(\nu) \circ t = t \circ \gamma_0(N)$ at each point on $\Sigma$. 

In the following, we fix sections $s$ and $t$ of $\text{\rm Hom}(S_0,\mathcal{S})$ satisfying the properties above. For each $\sigma \in S_0$ and each $\tau \in S_0$, the pair $(s\sigma,t\tau)$ defines a $\nabla^{\mathcal{W}}$-parallel section of the vector bundle $\mathcal{W}$.

\begin{definition}
We denote by $\mathcal{Y}$ the set of all pairs $(\xi,\omega) \in \R^n \times \Lambda^2 \R^{n,1}$ with the property that, for all $\sigma,\tau \in S_0$, 
\[(s \gamma_0(\xi)\sigma,-t \gamma_0(\xi) \tau) = i \, \gamma^{\mathcal{W}}(\omega) \, (s\sigma,t\tau)\] 
at each point on $M$. Let $\mathcal{Z} \subset \Lambda^2 \R^{n,1}$ denote the image of $\mathcal{Y}$ under the projection $\mathcal{Y} \to \Lambda^2 \R^{n,1}, \, (\xi,\omega) \mapsto \omega$.
\end{definition}

Clearly, $\mathcal{Y}$ is a linear subspace of $\R^n \times \Lambda^2 \R^{n,1}$, and $\mathcal{Z}$ is a linear subspace of $\Lambda^2 \R^{n,1}$.

\begin{lemma} 
\label{injectivity.1}
Suppose that $\xi \in \R^n$ satisfies $(\xi,0) \in \mathcal{Y}$. Then $\xi=0$.
\end{lemma}

\begin{proof}
Let us fix an arbitrary point $p \in M$. 
Since $(\xi,0) \in \mathcal{Y}$, we obtain 
\[
(s \gamma_0(\xi)\sigma,-t \gamma_0(\xi)\tau) = (0,0)
\] 
for all $\sigma,\tau \in S_0$. Since $s,t \in \text{\rm Hom}(S_0,\mathcal{S}|_p)$ are invertible, it follows that $\gamma_0(\xi) = 0 \in \text{\rm End}(S_0)$. Thus, $\xi=0$.
\end{proof}

\begin{lemma} 
\label{injectivity.2}
Suppose that $\omega \in \Lambda^2 \R^{n,1}$ satisfies $(0,\omega) \in \mathcal{Y}$. 
Then $\omega=0$.
\end{lemma}

\begin{proof}
Fix an arbitrary point $p \in M$. 
We may write $\omega = \rho + T \wedge \eta$, where $\rho \in \Lambda^2 T_p M$ and $\eta \in T_p M$. 
Since $\langle T,\eta \rangle_{\R^{n,1}} = 0$, we obtain $\gamma^{\mathcal{W}}(T \wedge \eta) = \gamma^{\mathcal{W}}(T) \, \gamma^{\mathcal{W}}(\eta)$. 
Since $(0,\omega) \in \mathcal{Y}$, we compute 
\begin{align*} 
(0,0) 
&= \gamma^{\mathcal{W}}(\omega) \, (s\sigma,t\tau) \\ 
&= \gamma^{\mathcal{W}}(\rho) \, (s\sigma,t\tau) + \gamma^{\mathcal{W}}(T) \, \gamma^{\mathcal{W}}(\eta) \, (s\sigma,t\tau) \\ 
&= (\gamma(\rho) s\sigma - i \, \gamma(\eta) s\sigma,\gamma(\rho) t\tau + i \, \gamma(\eta) t\tau) 
\end{align*}
for all $\sigma,\tau \in S_0$. 
Since $s,t \in \text{\rm Hom}(S_0,\mathcal{S}|_p)$ are invertible, it follows that 
\[
\gamma(\rho) - i \, \gamma(\eta) = 0 \in \text{\rm End}(\mathcal{S}|_p) \quad\text{ and }\quad \gamma(\rho) + i \, \gamma(\eta) = 0 \in \text{\rm End}(\mathcal{S}|_p).
\] 
This gives $\gamma(\rho) = 0 \in \text{\rm End}(\mathcal{S}|_p)$ and $\gamma(\eta) = 0 \in \text{\rm End}(\mathcal{S}|_p)$. 
Thus, $\rho=0$ and $\eta=0$. 
This finally implies $\omega=0$.
\end{proof}

\begin{lemma}
\label{y.preparation}
Let $p \in \Sigma$ be a boundary point. Then $(N(p),T(p) \wedge \nu(p)) \in \mathcal{Y}$.
\end{lemma}

\begin{proof}
Let us fix an arbitrary point $p \in \Sigma$, and let us fix elements $\sigma,\tau \in S_0$. The boundary conditions for $s$ and $t$ imply that $\gamma(\nu) \circ s = s \circ \gamma_0(N)$ and $\gamma(\nu) \circ t = t \circ \gamma_0(N)$ at the point $p$. Since $\langle T,\nu \rangle_{\R^{n,1}} = 0$, we obtain $\gamma^{\mathcal{W}}(T \wedge \nu) = \gamma^{\mathcal{W}}(T) \, \gamma^{\mathcal{W}}(\nu)$. This gives 
\begin{align*}
(s \gamma_0(N)\sigma,-t \gamma_0(N)\tau) 
&= (\gamma(\nu) s\sigma,-\gamma(\nu) t\tau) \\ 
&= i \, \gamma^{\mathcal{W}}(T) \, \gamma^{\mathcal{W}}(\nu) \, (s\sigma,t\tau) \\ 
&= i \, \gamma^{\mathcal{W}}(T \wedge \nu) \, (s\sigma,t\tau) 
\end{align*}
at the point $p$. For abbreviation, let $\xi = N(p) \in \R^n$, $a = T(p) \in \R^{n,1}$, and $b = \nu(p) \in \R^{n,1}$. Then 
\[(s \gamma_0(\xi)\sigma,-t \gamma_0(\xi)\tau) = i \, \gamma^{\mathcal{W}}(a \wedge b) \, (s\sigma,t\tau)\] 
at the point $p$. Note that $(s \gamma_0(\xi)\sigma,-t \gamma_0(\xi)\tau)$ is a $\nabla^{\mathcal{W}}$-parallel section of the vector bundle $\mathcal{W}$. Moreover, it follows from Lemma~\ref{compatibility} that $\gamma^{\mathcal{W}}(a \wedge b) \, (s\sigma,t\tau)$ is a $\nabla^{\mathcal{W}}$-parallel section of the vector bundle $\mathcal{W}$. Putting these facts together, we conclude that 
\[(s \gamma_0(\xi)\sigma,-t \gamma_0(\xi)\tau) = i \, \gamma^{\mathcal{W}}(a \wedge b) \, (s\sigma,t\tau)\] 
at each point in $M$. Since $\sigma,\tau \in S_0$ are arbitrary, it follows that $(\xi,a \wedge b) \in \mathcal{Y}$. 
\end{proof}

\begin{proposition} 
\label{y}
We have 
\[
\mathcal{Y} = \{\R \cdot (N(p),T(p) \wedge \nu(p)) \colon p \in \Sigma\}.
\] 
The projection $\mathcal{Y} \to \R^n, \, (\xi,\omega) \mapsto \xi$ is bijective. In particular, $\mathcal{Y}$ has dimension $n$.
\end{proposition}

\begin{proof} 
We consider the set 
\[\mathcal{Y}_0 = \{\R \cdot (N(p),T(p) \wedge \nu(p)) \colon p \in \Sigma\} \subset \R^n \times \Lambda^2 \R^{n,1}.\] 
It follows from Lemma~\ref{y.preparation} that $\mathcal{Y}_0 \subset \mathcal{Y}$. By Lemma~\ref{injectivity.2}, the projection $\mathcal{Y} \to \R^n, \, (\xi,\omega) \mapsto \xi$ is injective. On the other hand, since the map $N \colon \Sigma \to S^{n-1}$ is surjective, it follows that the projection $\mathcal{Y}_0 \to \R^n, \, (\xi,\omega) \mapsto \xi$ is surjective. Putting these facts together, we conclude that $\mathcal{Y}_0 = \mathcal{Y}$ and the projection $\mathcal{Y} \to \R^n, \, (\xi,\omega) \mapsto \xi$ is bijective. 
\end{proof}

\begin{corollary} 
\label{z}
We have $\mathcal{Z} = \{\R \cdot (T(p) \wedge \nu(p)) \colon p \in \Sigma\}$. Moreover, $\mathcal{Z}$ has dimension $n$.
\end{corollary}

\begin{proof} 
The first statement follows immediately from Proposition~\ref{y}. The second statement follows from the fact that $\mathcal{Y}$ has dimension $n$ and the projection $\mathcal{Y} \to \Lambda^2 \R^{n,1}, \, (\xi,\omega) \mapsto \omega$ is injective (see Lemma~\ref{injectivity.1}). 
\end{proof}

\begin{proposition} 
\label{basis}
There exists a basis $\{b_0,b_1,\hdots,b_n\}$ of $\R^{n,1}$ such that $b_0$ is timelike and $\mathcal{Z} = \text{\rm span}\{b_0 \wedge b_1,b_0 \wedge b_2,\hdots,b_0 \wedge b_n\}$. 
\end{proposition}

\begin{proof} 
It follows from Corollary~\ref{z} that the subspace $\mathcal{Z} \subset \Lambda^2 \R^{n,1}$ is decomposable in the sense of \cite{Ghorpade-Patil-Pillai}. 
(See \cite{Ghorpade-Patil-Pillai}*{p.~55} for the definition.) 
Applying \cite{Ghorpade-Patil-Pillai}*{Theorem 6} with $l=2$ and $r=n$, we conclude that the subspace $\mathcal{Z} \subset \Lambda^2 \R^{n,1}$ is close in the sense of \cite{Ghorpade-Patil-Pillai}. 
(See \cite{Ghorpade-Patil-Pillai}*{p.~57} for the definition.) 
Thus, there are two possibilities.
\medskip

\textit{Case 1:} 
The subspace $\mathcal{Z} \subset \Lambda^2 \R^{n,1}$ is close of Type I (cf. \cite{Ghorpade-Patil-Pillai}, p.~57). 
In this case, there exists a basis $\{b_0,b_1,\hdots,b_n\}$ of $\R^{n,1}$ with the property that 
\[
\mathcal{Z} = \text{\rm span}\{b_0 \wedge b_1,b_0 \wedge b_2,\hdots,b_0 \wedge b_n\}.
\]
We claim that $b_0$ is timelike. Suppose that $b_0$ is spacelike or null. We can find a spacelike vector $\tilde{b} \in \R^{n,1}$ such that $\langle b_0,\tilde{b} \rangle_{\R^{n,1}} = 0$. This implies that $b_0$ and $\tilde{b}$ are linearly independent. 
Moreover, $b_0 \wedge \tilde{b} \in \mathcal{Z}$. 
By Corollary~\ref{z}, we can find a point $p \in \Sigma$ with the property that $b_0 \wedge \tilde{b}$ is a non-zero scalar multiple of $T(p) \wedge \nu(p)$. 
In particular, $\text{\rm span}\{b_0,\tilde{b}\} = \text{\rm span}\{T(p),\nu(p)\}$. 
This contradicts the fact that $\text{\rm span}\{b_0,\tilde{b}\}$ does not contain a timelike vector. 
\medskip

\textit{Case 2:} 
The subspace $\mathcal{Z} \subset \Lambda^2 \R^{n,1}$ is close of Type II, but not close of Type I (cf. \cite{Ghorpade-Patil-Pillai}*{p.~57}). 
In this case, $n=3$. 
Moreover, there exists a collection of three linearly independent vectors $b_0,b_1,b_2 \in \R^{3,1}$ with the property that $\mathcal{Z} = \text{\rm span}\{b_0 \wedge b_1,b_1 \wedge b_2,b_2 \wedge b_0\}$. We can find two linearly independent vectors $\tilde{b}_1,\tilde{b}_2 \in \text{\rm span}\{b_0,b_1,b_2\}$ such that $\text{\rm span}\{\tilde{b}_1,\tilde{b}_2\}$ is spacelike. 
Then $\tilde{b}_1 \wedge \tilde{b}_2 \in \mathcal{Z}$. 
By Corollary~\ref{z}, we can find a point $p \in \Sigma$ with the property that $\tilde{b}_1 \wedge \tilde{b}_2$ is a non-zero scalar multiple of $T(p) \wedge \nu(p)$. 
In particular, 
\[
\text{\rm span}\{\tilde{b}_1,\tilde{b}_2\} = \text{\rm span}\{T(p),\nu(p)\}.
\]
This contradicts the fact that $\text{\rm span}\{\tilde{b}_1,\tilde{b}_2\}$ does not contain a timelike vector. 
This case cannot occur.
\end{proof}

By scaling, we may assume that $\langle b_0,b_0 \rangle_{\R^{n,1}} = -1$. By applying a Lorentz transformation, we may assume that $b_0 = (1,0,\hdots,0) \in \R^{n,1}$. 

For each point $p \in \Sigma$, Lemma~\ref{y.preparation} implies that 
\[
T(p) \wedge \nu(p) \in \mathcal{Z} = \text{\rm span}\{b_0 \wedge b_1,b_0 \wedge b_2,\hdots,b_0 \wedge b_n\};
\] 
hence $b_0 \in \text{\rm span}\{T(p),\nu(p)\}$. 
Consequently, we have $\langle b_0,X \rangle_{\R^{n,1}} = 0$ for every tangent vector $X \in T\Sigma$. 
In particular, each connected component of $\Sigma$ is contained in a hyperplane in $\R^{n,1}$ which is orthogonal to $b_0$. 

For each point $p \in \Sigma$, there is a unique spacelike unit vector $B(p) \in \R^{n,1}$ such that $\langle b_0,B(p) \rangle_{\R^{n,1}} = 0$ and $T(p) \wedge \nu(p) = b_0 \wedge B(p)$. 
This defines a smooth map $B$ from $\Sigma$ to the $(n-1)$-dimensional unit sphere 
\[
S^{n-1} = \{b \in \R^{n,1} \colon \langle b_0,b \rangle_{\R^{n,1}} = 0, \, \langle b,b \rangle_{\R^{n,1}} = 1\}.
\]
The map $B$ can be interpreted as the unit normal to $\Sigma$ within a hyperplane orthogonal to $b_0$. 

In view of Lemma~\ref{injectivity.2}, we can find a linear transformation $A \colon \R^n \to \mathcal{Z}$ such that 
\[
\mathcal{Y} \subset \{(\xi,A\xi) \colon \xi \in \R^n\}.
\] 
Using Lemma~\ref{y.preparation}, it follows that the map $\Sigma \to \mathcal{Z}, \, p \mapsto T(p) \wedge \nu(p)$ is the composition of the map $N \colon \Sigma \to S^{n-1}$ with the linear transformation $A \colon \R^n \to \mathcal{Z}$. 
This implies that the map $B \colon \Sigma \to S^{n-1}$ is the composition of the map $N \colon \Sigma \to S^{n-1}$ with a linear isometry of $\R^n$. 
In particular, $\|dB\|_{\text{\rm tr}} = \|dN\|_{\text{\rm tr}}$ at each point on $\Sigma$.

\begin{proposition}
\label{dichotomy}
Suppose that $p$ is a point on the boundary $\Sigma$ with $\langle b_0,\nu \rangle_{\R^{n,1}} \neq 0$. Then $h=0$, $dN=0$, and $dB=0$ at the point $p$. Moreover, the second fundamental form of $\Sigma$ in $\R^{n,1}$ vanishes at the point $p$.
\end{proposition}

\begin{proof}
Let $\{e_1,\hdots,e_{n-1}\}$ denote an orthonormal basis of $T_p \Sigma$. Differentiating the identity $T \wedge \nu = b_0 \wedge B$ gives 
\[\sum_{j=1}^{n-1} h(e_i,e_j) \, T \wedge e_j + \sum_{j=1}^{n-1} q(e_i,e_j) \, e_j \wedge \nu = b_0 \wedge dB(e_i)\] 
for each $i=1,\hdots,n-1$. Proposition~\ref{prop:prelim} implies that $q(e_i,e_j) = 0$ for $i,j=1,\hdots,n-1$. This gives 
\begin{align*} 
0 
&= \sum_{j=1}^{n-1} h(e_i,e_j) \, \langle T \wedge e_j,e_k \wedge \nu \rangle_{\R^{n,1}} + \sum_{j=1}^{n-1} q(e_i,e_j) \, \langle e_j \wedge \nu,e_k \wedge \nu \rangle_{\R^{n,1}} \\ 
&= \langle b_0 \wedge dB(e_i),e_k \wedge \nu \rangle_{\R^{n,1}} \\ 
&= -\langle b_0,\nu \rangle_{\R^{n,1}} \, \langle dB(e_i),e_k \rangle_{\R^{n,1}} 
\end{align*} 
for $i,k=1,\hdots,n-1$. By assumption, $\langle b_0,\nu \rangle_{\R^{n,1}} \neq 0$ at the point $p$. This implies $dB=0$ at the point $p$. Consequently, $dN=0$ at the point $p$. Using Proposition~\ref{prop:prelim}, we conclude that $h=0$ at the point $p$. Since $q^{\text{\rm tan}} = 0$ by Proposition~\ref{prop:prelim}, it follows that the second fundamental form of $\Sigma$ in $\R^{n,1}$ vanishes at the point $p$.
\end{proof}

\begin{proposition}
\label{diffeo.to.ball}
Let $\pi$ denote the projection from $\R^{n,1}$ to 
\[
\R^n = \{b \in \R^{n,1} \colon \langle b_0,b \rangle_{\R^{n,1}} = 0\}.
\]
The composition of the isometric immersion 
\[
F \colon (M,g) \to \R^{n,1}
\] 
with the projection $\pi \colon \R^{n,1} \to \R^n$ is a diffeomorphism from $M$ onto a  compact convex domain $K \subset \R^n$ with smooth boundary. 
In particular, $F$ is an isometric embedding onto the graph of a smooth real valued function defined on $K$.
\end{proposition}

\begin{proof}
We define a Riemannian metric $g_0$ on $M$ by 
\[g_0(X,Y) = g(X,Y) + \langle b_0,X \rangle_{\R^{n,1}} \, \langle b_0,Y \rangle_{\R^{n,1}}\] 
for $X,Y \in TM$. In other words, $g_0$ is the pull-back of the Euclidean metric under the immersion $\pi \circ F \colon M \to \R^n$.
In particular, the metric $g_0$ is flat. It follows from Proposition~\ref{dichotomy} that the second fundamental form of $\partial M$ with respect to the metric $g_0$ coincides with the second fundamental form of $\partial M$ with respect to the metric $g$. Using Proposition~\ref{prop:prelim}, we conclude that the second fundamental form of $\partial M$ with respect to the metric $g_0$ is weakly positive definite at each point on $\partial M$. Moreover, since $N$ has non-zero degree, we can find a point on $\partial M$ where the Jacobian determinant of $N$ is non-zero, and at that point the second fundamental of $\partial M$ with respect to $g_0$ is positive definite by Proposition~\ref{prop:prelim}. Theorem~\ref{thm:completeconvex} now implies that there exists an isometry of $(M,g_0)$ onto a compact convex domain in $\R^n$ with smooth boundary. 
The isometric immersion $\pi \circ F \colon (M,g_0) \to \R^n$ coincides with this isometry up to a rigid motion of $\R^n$. 
This implies the assertion.
\end{proof}

This completes the proof of Theorem~\ref{rigidity.of.smooth.domains} in the special case when $M$ is simply connected.

\subsection{Proof of Theorem~\ref{rigidity.of.smooth.domains} in the general case} 

In this final subsection, we treat the case when $M$ is not simply connected. 
In this case, the isometric immersion $F \colon M \to \R^{n,1}$ is only locally defined, but its differential is globally defined. 
This is sufficient to define the bundle $\mathcal{W}$, the connection $\nabla^{\mathcal{W}}$, and the map $\gamma^{\mathcal{W}} \colon \R^{n,1} \to \text{\rm End}(\mathcal{W}|_p)$. 
The arguments in Section~\ref{simply.connected} now imply that there exists a vector $b_0 \in \R^{n,1}$ such that 
\[
\langle b_0,b_0 \rangle_{\R^{n,1}} = -1 \quad \text{ and } \quad \langle b_0,dF(X) \rangle_{\R^{n,1}} = 0
\]
for every tangent vector $X \in T\Sigma$. 

As above, we define a Riemannian metric $g_0$ on $M$ by 
\[g_0(X,Y) = g(X,Y) + \langle b_0,dF(X) \rangle_{\R^{n,1}} \, \langle b_0,dF(Y) \rangle_{\R^{n,1}}\] 
for $X,Y \in TM$. 
Note that $g_0$ is a globally defined metric on $M$. As in Section~\ref{simply.connected}, we see that $g_0$ is flat. Moreover, the second fundamental form of $\partial M$ with respect to the metric $g_0$ coincides with the second fundamental form of $\partial M$ with respect to the metric $g$. In particular, the second fundamental form of $\partial M$ with respect to the metric $g_0$ is weakly positive definite at each point on $\partial M$, and we can find a point on $\partial M$ where the second fundamental of $\partial M$ with respect to $g_0$ is positive definite. Using Theorem~\ref{thm:completeconvex}, it follows that $M$ is diffeomorphic to a ball. 
This contradicts the assumption that $M$ is not simply connected. 

\section{Approximating a compact convex polytope by smooth domains}

Throughout this section, we fix an integer $n \geq 2$.
We further assume that $\Omega \subset \R^n$ is a compact convex polytope with nonempty interior. 
We write $\Omega = \bigcap_{i \in I} \{u_i \leq 0\}$, where $I$ is a finite set and $u_i$, $i \in I$, is a collection of linear functions on $\R^n$. After eliminating redundant inequalities, we may assume that the following condition is satisfied. 

\begin{assumption}
\label{no.redundant.inequalities}
For each $i_0 \in I$, the following set is nonempty:
\[
\{u_{i_0} > 0\} \cap \bigcap_{i \in I \setminus \{i_0\}} \{u_i \leq 0\}.
\] 
\end{assumption}

Let $g$ be a Riemannian metric which is defined on an open set containing $\Omega$, and let $q$ be a symmetric $(0,2)$-tensor. 
For each $i \in I$, $\nabla u_i$ will denote the gradient of $u_i$ with respect to the metric $g$, $D^2 u_i$ will denote the Hessian of $u_i$ with respect to the metric $g$, $|\nabla u_i|_g$ will denote the norm of the gradient of $u_i$ with respect to the metric $g$, and $\nu_i = {\nabla u_i}/{|\nabla u_i|_g}$ will denote the unit normal vector field, with respect to the metric $g$, to the level sets of $u_i$. 
For each $i \in I$, we denote by $N_i \in S^{n-1}$ the outward-pointing unit normal vector to the halfspace $\{u_i \leq 0\}$ with respect to the Euclidean metric.

For each $\lambda>0$, we define 
\[\Omega_\lambda = \bigg \{ \sum_{i \in I} e^{\lambda u_i} \leq 1 \bigg \} \subset \Omega.\] 
If $\lambda$ is sufficiently large, then $\Omega_\lambda$ is a compact convex domain in $\R^n$ with smooth boundary $\Sigma_\lambda = \partial \Omega_\lambda$. 
The sets $\Omega_\lambda$ form an increasing family of sets. 
Moreover, 
\[\bigcup_{\lambda>\lambda_0} \Omega_\lambda = \bigcap_{i \in I} \{u_i < 0\}.\] 

In the remainder of this section, $|\cdot |_g$ will denote the norm taken with respect to the metric $g$, and $|\cdot |$ will denote the norm taken with respect to the Euclidean metric. 
The outward-pointing unit normal vector to the domain $\Omega_\lambda$ with respect to the metric $g$ is given by 
\[\nu = \frac{\sum_{i \in I} e^{\lambda u_i} \, \nabla u_i}{\big | \sum_{i \in I} e^{\lambda u_i} \, \nabla u_i \big |_g} = \frac{\sum_{i \in I} e^{\lambda u_i} \, |\nabla u_i|_g \, \nu_i}{\big | \sum_{i \in I} e^{\lambda u_i} \, |\nabla u_i|_g \, \nu_i \big |_g}.\] 
We define a map $N \colon \Sigma_\lambda \to S^{n-1}$ by  
\begin{equation} \label{defn:N} 
N = \frac{\sum_{i \in I} e^{\lambda u_i} \, |\nabla u_i|_g \, N_i}{\big | \sum_{i \in I} e^{\lambda u_i} \, |\nabla u_i|_g \, N_i \big |}.
\end{equation} 
The map $N \colon \Sigma_\lambda \to S^{n-1}$ is homotopic to the Gauss map of $\Sigma_\lambda$ with respect to the Euclidean metric. 
A homotopy can be constructed by deforming $g$ to the Euclidean metric. 
In particular, the map $N \colon \Sigma_\lambda \to S^{n-1}$ has degree $1$.

\begin{proposition}[cf. \cite{Brendle}*{Proposition~3.9}]
\label{formula.for.V}
Let $x \in \Sigma_\lambda$. 
Let $\pi\colon T_x \Omega \to T_x \Omega$ denote the orthogonal projection to the orthogonal complement of $\nu$, and let $P\colon \R^n \to \R^n$ denote the orthogonal projection to the orthogonal complement of $N$. Then $H - \|dN\|_\tr \geq V_\lambda$, where the function $V_\lambda\colon \Sigma_\lambda \to \R$ is defined by 
\begin{align*} 
V_\lambda &= \lambda \, \frac{\sum_{i \in I} e^{\lambda u_i} \, |\nabla u_i|_g^2 \, |\pi(\nu_i)|_g^2}{\big | \sum_{i \in I} e^{\lambda u_i} \, |\nabla u_i|_g \, \nu_i \big |_g} - \lambda \, \frac{\sum_{i \in I} e^{\lambda u_i} \, |\nabla u_i|_g^2 \, |\pi(\nu_i)|_g \, |P(N_i)|}{\big | \sum_{i \in I} e^{\lambda u_i} \, |\nabla u_i|_g \, N_i \big |} \\ 
&\quad + \frac{\sum_{i \in I} e^{\lambda u_i} \, (\Delta u_i - (D^2 u_i)(\nu,\nu))}{\big | \sum_{i \in I} e^{\lambda u_i} \, |\nabla u_i|_g \, \nu_i \big |_g} - \frac{\sum_{i \in I} e^{\lambda u_i} \, |\nabla(|\nabla u_i|_g)|_g \, |P(N_i)|}{\big | \sum_{i \in I} e^{\lambda u_i} \, |\nabla u_i|_g \, N_i \big |}.
\end{align*} 
\end{proposition}

We define a function $W_\lambda\colon \Sigma_\lambda \to \R$ by $W_\lambda = V_\lambda - |\tr(q)-q(\nu,\nu)|$.

\begin{proposition}
\label{negative.part.of.W.small.scale}
Suppose that the Matching Angle Hypothesis is satisfied. 
Then 
\[
\sup_{\Sigma_\lambda} \max \{-W_\lambda,0\} \leq o(\lambda) \quad\text{ as } \lambda \to \infty.
\]
\end{proposition}

\begin{proof}
It follows from \cite{Brendle}*{Proposition~3.10} that $\sup_{\Sigma_\lambda} \max \{-V_\lambda,0\} \leq o(\lambda)$ as $\lambda \to \infty$. Since $W_\lambda - V_\lambda$ is uniformly bounded, the assertion follows.
\end{proof}

In the remaining part of this section, we will estimate the $L^\sigma$-norm $\max \{-W_\lambda,0\}$ on $\Sigma_\lambda \cap B_r(p)$, where $\sigma \in [1,\tfrac{3}{2})$ is a fixed exponent and $B_r(p)$ denotes a Euclidean ball of radius $r$.

\begin{definition}
Consider three pairwise distinct elements $i_1,i_2,i_3 \in I$. We denote by $G_\lambda^{(i_1,i_2,i_3)}$ the set of all points $x \in \Sigma_\lambda$ with the property that $u_{i_1}(x) \geq u_{i_2}(x)\geq u_{i_3}(x)$ and $u_{i_3}(x) \geq u_i(x)$ for each $i \in I \setminus \{i_1,i_2,i_3\}$. 
\end{definition}

\begin{lemma}
\label{step.1}
For each $i \in I$, we assume that $H - |\tr(q) - q(\nu_i,\nu_i)| \geq 0$ at each point in $\Omega \cap \{u_i=0\}$, where $H$ denotes the mean curvature of the hypersurface $\{u_i=0\}$ with respect to the metric $g$. Let us fix an exponent $\sigma \in [1,\tfrac{3}{2})$, and let $B_r(p)$ denote a Euclidean ball of radius $r \leq 1$. If $\lambda r$ is sufficiently large, then 
\[\bigg ( r^{\sigma+1-n} \int_{G_\lambda^{(i_1,i_2,i_3)} \cap \{u_{i_2} \leq -\lambda^{-\frac{7}{8}} r^{\frac{1}{8}}\} \cap B_r(p)} \max \{-W_\lambda,0\}^\sigma \bigg )^{\frac{1}{\sigma}} \leq C\lambda r \, e^{-(\lambda r)^{\frac{1}{8}}}\] 
for all pairwise distinct elements $i_1,i_2,i_3 \in I$. 
\end{lemma}

\begin{proof}
We adapt the proof of \cite{Brendle}*{Lemma~3.13}. 
Let us consider an arbitrary point 
\[
x \in G_\lambda^{(i_1,i_2,i_3)} \quad\text{ with } u_{i_2}(x) \leq -\lambda^{-\frac{7}{8}} r^{\frac{1}{8}}.
\]
By definition of $G_\lambda^{(i_1,i_2,i_3)}$, it follows that $u_i(x) \leq -\lambda^{-\frac{7}{8}} r^{\frac{1}{8}}$ for all $i \in I \setminus \{i_1\}$. 
Using the identity 
\[
\sum_{i \in I} e^{\lambda u_i(x)} = 1,
\]
we obtain $e^{\lambda u_{i_1}(x)} \geq 1 - C \, e^{-(\lambda r)^{\frac{1}{8}}}$, hence $u_{i_1}(x) \geq -C\lambda^{-1} \, e^{-(\lambda r)^{\frac{1}{8}}}$. 
Arguing as in \cite{Brendle}, we obtain 
\[V_\lambda \geq \frac{\Delta u_{i_1} - (D^2 u_{i_1})(\nu_{i_1},\nu_{i_1})}{|\nabla u_{i_1}|_g} - C \lambda \, e^{-(\lambda r)^{\frac{1}{8}}}\] 
at the point $x$. 
Moreover, $|\nu-\nu_{i_1}|_g \leq C \, e^{-(\lambda r)^{\frac{1}{8}}}$ at the point $x$. 
This implies 
\[|\tr(q) - q(\nu,\nu)| \leq |\tr(q) - q(\nu_{i_1},\nu_{i_1})| + C \, e^{-(\lambda r)^{\frac{1}{8}}}\] 
at the point $x$. 
Putting these facts together, we obtain 
\begin{align*}
W_\lambda 
&= V_\lambda - |\tr(q) - q(\nu,\nu)| \\ 
&\geq \frac{\Delta u_{i_1} - (D^2 u_{i_1})(\nu_{i_1},\nu_{i_1})}{|\nabla u_{i_1}|_g} - |\tr(q) - q(\nu_{i_1},\nu_{i_1})| - C \lambda \, e^{-(\lambda r)^{\frac{1}{8}}} 
\end{align*}
at the point $x$. 
Since $u_{i_1}(x) \geq -C\lambda^{-1} \, e^{-(\lambda r)^{\frac{1}{8}}}$ and $u_i(x) \leq -\lambda^{-\frac{7}{8}} r^{\frac{1}{8}}$ for all $i \in I \setminus \{i_1\}$, we can find a point $y \in \Omega$ such that 
\[
u_{i_1}(y) = 0 \quad \text{ and } \quad d(x,y) \leq C\lambda^{-1} \, e^{-(\lambda r)^{\frac{1}{8}}}.
\]
By assumption, $H - |\tr(q) - q(\nu_{i_1},\nu_{i_1})| \geq 0$ at the point $y$, where $H$ denotes the mean curvature of the hypersurface $\{u_{i_1}=0\}$ with respect to the metric $g$. 
This implies 
\[\frac{\Delta u_{i_1} - (D^2 u_{i_1})(\nu_{i_1},\nu_{i_1})}{|\nabla u_{i_1}|_g} - |\tr(q) - q(\nu_{i_1},\nu_{i_1})| \geq 0\] 
at the point $y$. 
Consequently, 
\[\frac{\Delta u_{i_1} - (D^2 u_{i_1})(\nu_{i_1},\nu_{i_1})}{|\nabla u_{i_1}|_g} - |\tr(q) - q(\nu_{i_1},\nu_{i_1})| \geq -C \, d(x,y)\] 
at the point $x$. 
Putting these facts together, we conclude that 
\[W_\lambda(x) \geq -C\lambda \, e^{-(\lambda r)^{\frac{1}{8}}}\] 
for each point 
\[
x \in G_\lambda^{(i_1,i_2,i_3)} \cap \{u_{i_2} \leq -\lambda^{-\frac{7}{8}} r^{\frac{1}{8}}\}.
\]
On the other hand, the intersection $\Sigma_\lambda \cap B_r(p)$ has area at most $C r^{n-1}$ (see \cite{Brendle}*{Lemma~3.11}). Consequently, 
\[\bigg ( r^{\sigma+1-n} \int_{G_\lambda^{(i_1,i_2,i_3)} \cap \{u_{i_2} \leq -\lambda^{-\frac{7}{8}} r^{\frac{1}{8}}\} \cap B_r(p)} \max \{-W_\lambda,0\}^\sigma \bigg )^{\frac{1}{\sigma}} \leq C\lambda r \, e^{-(\lambda r)^{\frac{1}{8}}}.\] 
This completes the proof of Lemma~\ref{step.1}.
\end{proof}

\begin{lemma} 
\label{step.2}
Assume that the Matching Angle Hypothesis holds.  Let us fix an exponent $\sigma \in [1,\tfrac{3}{2})$, and let $B_r(p)$ denote a Euclidean ball of radius $r \leq 1$. If $\lambda r$ is sufficiently large, then 
\begin{align*} 
\bigg ( r^{\sigma+1-n} &\int_{G_\lambda^{(i_1,i_2,i_3)} \cap \{u_{i_2} \geq -\lambda^{-\frac{7}{8}} r^{\frac{1}{8}}\} \cap \{u_{i_3} \leq -\lambda^{-\frac{3}{4}} r^{\frac{1}{4}}\} \cap B_r(p)} \max \{-W_\lambda,0\}^\sigma \bigg )^{\frac{1}{\sigma}} 
\leq C \, (\lambda r)^{\frac{1}{8}-\frac{7}{8\sigma}} 
\end{align*}
for all pairwise distinct elements $i_1,i_2,i_3 \in I$. 
\end{lemma}

\begin{proof} 
The corresponding estimate for $V_\lambda$ follows from \cite{Brendle}*{Lemma~3.14}. Moreover, it is shown in \cite{Brendle} that the set $G_\lambda^{(i_1,i_2,i_3)} \cap \{u_{i_2} \geq -\lambda^{-\frac{7}{8}} r^{\frac{1}{8}}\} \cap \{u_{i_3} \leq -\lambda^{-\frac{3}{4}} r^{\frac{1}{4}}\}$ has area at most $C \, (\lambda r)^{-\frac{7}{8}} \, r^{n-1}$. Since $|W_\lambda - V_\lambda| \leq C \leq C \lambda^{\frac{1}{8}} r^{-\frac{7}{8}}$, the assertion follows.
\end{proof}

\begin{lemma} 
\label{step.3}
Let us fix an exponent $\sigma \in [1,\tfrac{3}{2})$, and let $B_r(p)$ denote a Euclidean ball of radius $r \leq 1$. 
If $\lambda r$ is sufficiently large, then 
\[
\bigg ( r^{\sigma+1-n} \int_{G_\lambda^{(i_1,i_2,i_3)} \cap \{u_{i_3} \geq -\lambda^{-\frac{3}{4}} r^{\frac{1}{4}}\} \cap B_r(p)} \max \{-W_\lambda,0\}^\sigma \bigg )^{\frac{1}{\sigma}}  \leq C \, (\lambda r)^{1-\frac{3}{2\sigma}}
\] 
for all pairwise distinct elements $i_1,i_2,i_3 \in I$. 
\end{lemma}

\begin{proof}
The corresponding estimate for $V_\lambda$ follows from \cite{Brendle}*{Lemma~3.15}. 
Moreover, it is shown in \cite{Brendle} that the set $G_\lambda^{(i_1,i_2,i_3)} \cap \{u_{i_3} \geq -\lambda^{-\frac{3}{4}} r^{\frac{1}{4}}\} \cap B_r(p)$ has area at most $C \, (\lambda r)^{-\frac{3}{2}} \, r^{n-1}$. Since $|W_\lambda - V_\lambda| \leq C \leq C \lambda$, the assertion follows.
\end{proof}

\begin{proposition}
\label{negative.part.of.W.large.scale}
For each $i \in I$, we assume that $H - |\tr(q) - q(\nu_i,\nu_i)| \geq 0$ at each point in $\Omega \cap \{u_i=0\}$, where $H$ denotes the mean curvature of the hypersurface $\{u_i=0\}$ with respect to the metric $g$. Moreover, we assume that the Matching Angle Hypothesis is satisfied. 
Let us fix an exponent $\sigma \in [1,\tfrac{3}{2})$, and let $B_r(p)$ denote a Euclidean ball of radius $r \leq 1$. 
If $\lambda r$ is sufficiently large, then 
\begin{align*} 
\bigg ( r^{\sigma+1-n}& \int_{\Sigma_\lambda \cap B_r(p)} \max \{-W_\lambda,0\}^\sigma \bigg )^{\frac{1}{\sigma}} 
\leq C\lambda r \, e^{-(\lambda r)^{\frac{1}{8}}} + C \, (\lambda r)^{\frac{1}{8}-\frac{7}{8\sigma}} + C \, (\lambda r)^{1-\frac{3}{2\sigma}}. 
\end{align*}
\end{proposition}

\begin{proof}
Combining Lemma~\ref{step.1}, Lemma~\ref{step.2}, and Lemma~\ref{step.3}, we conclude that 
\begin{align*} 
\bigg ( r^{\sigma+1-n} &\int_{G_\lambda^{(i_1,i_2,i_3)} \cap B_r(p)} \max \{-W_\lambda,0\}^\sigma \bigg )^{\frac{1}{\sigma}} 
\leq C\lambda r \, e^{-(\lambda r)^{\frac{1}{8}}} + C \, (\lambda r)^{\frac{1}{8}-\frac{7}{8\sigma}} + C \, (\lambda r)^{1-\frac{3}{2\sigma}} 
\end{align*}
for all pairwise distinct elements $i_1,i_2,i_3 \in I$. 
On the other hand, $\Sigma_\lambda = \bigcup_{i_1,i_2,i_3} G_\lambda^{(i_1,i_2,i_3)}$, where the union is taken over all pairwise distinct elements $i_1,i_2,i_3 \in I$. 
Hence, the assertion follows by summation over $i_1,i_2,i_3$. 
This completes the proof of Proposition~\ref{negative.part.of.W.large.scale}. 
\end{proof}

\begin{corollary}
\label{negative.part.of.W}
For each $i \in I$, we assume that $H - |\tr(q) - q(\nu_i,\nu_i)| \geq 0$ at each point in $\Omega \cap \{u_i=0\}$, where $H$ denotes the mean curvature of the hypersurface $\{u_i=0\}$ with respect to the metric $g$. Moreover, we assume that the Matching Angle Hypothesis is satisfied. 
Let us fix an exponent $\sigma \in [1,\tfrac{3}{2})$. Then 
\[
\sup_{p \in \R^n} \sup_{r \leq 1} \bigg ( r^{\sigma+1-n} \int_{\Sigma_\lambda \cap B_r(p)} \max \{-W_\lambda,0\}^\sigma \bigg )^{\frac{1}{\sigma}} \to 0 \quad\text{ as }\lambda \to \infty.
\] 
\end{corollary}

\begin{proof}
Let us consider an arbitrary sequence $\lambda_l \to \infty$. 
By Proposition~\ref{negative.part.of.W.small.scale}, we can find a sequence of positive real numbers $\delta_l \to 0$ such that 
\[
(\delta_l \lambda_l)^{-1} \sup_{\Sigma_{\lambda_l}} \max \{-W_{\lambda_l},0\} \to 0 \quad\text{ as } l \to \infty.
\] 
This implies 
\begin{align*} 
\sup_{p \in \R^n} \sup_{r \leq (\delta_l \lambda_l)^{-1}} \bigg ( r^{\sigma+1-n} &\int_{\Sigma_{\lambda_l} \cap B_r(p)} \max \{-W_{\lambda_l},0\}^\sigma \bigg )^{\frac{1}{\sigma}} \\ 
&\leq (\delta_l \lambda_l)^{-1} \sup_{\Sigma_{\lambda_l}} \max \{-W_{\lambda_l},0\} \to 0 \quad\text{ as } l \to \infty.
\end{align*} 
On the other hand, it follows from Proposition~\ref{negative.part.of.W.large.scale} that 
\[
\sup_{p \in \R^n} \sup_{(\delta_l \lambda_l)^{-1} \leq r \leq 1} \bigg ( r^{\sigma+1-n} \int_{\Sigma_{\lambda_l} \cap B_r(p)} \max \{-W_{\lambda_l},0\}^\sigma \bigg )^{\frac{1}{\sigma}} \to 0 \quad \text{ as } l \to \infty.
\] 
Putting these facts together, the assertion follows.
\end{proof}

\section{Proof of Theorem~\ref{rigidity.of.polytopes}}

Throughout this section, we assume that $\Omega$, $g$, $q$ satisfy the assumptions in Theorem~\ref{rigidity.of.polytopes}.
In particular, we assume that $\Omega$ satisfies the Matching Angle hypothesis. 
We consider the spinor bundle $\mathcal{S}$ and the twisted spinor bundle $\mathcal{E} = \mathcal{S} \otimes S_0^*$ over an open set containing $\Omega$. 

Let $U$ denote a Euclidean ball such that the closure of $U$ is contained in the interior of $\Omega$. 
Consider a sequence $\lambda_l\to\infty$. Given  $l$, let $N^{(l)} \colon \Sigma_{\lambda_l} \to S^{n-1}$ denote the map defined in \eqref{defn:N}. Recall that $N^{(l)}$ has degree $1$. Note that $U \subset \Omega_{\lambda_l}$ if $l$ is sufficiently large. 

\begin{proposition}\label{sobolev.type.inequality}
Suppose that $K>0$ is a given positive real number. Then there exists a uniform constant $C$ (independent of $l$) such that
\[\int_{\Omega_{\lambda_l}} F^2 \leq C \int_{\Omega_{\lambda_l}} (|\nabla F| -KF)_+^2 + C\int_U F^2 \]
for every nonnegative smooth function $F$ on $\Omega_{\lambda_l}$.
\end{proposition} 

\begin{proof}
Note that the hypersurface $\Sigma_{\lambda_l} = \partial\Omega_{\lambda_l}$ can be written as a radial graph with bounded slope. From this, it is easy to see that $\Omega_{\lambda_l}$ is bi-Lipschitz equivalent to the Euclidean unit ball, with constants that are independent of $l$. The assertion now follows from Proposition~\ref{sobolev.type.ineq.fixed.domain}.  
\end{proof}

\begin{proposition}\label{sobolev.trace}
There exists a uniform constant $C$ (independent of $l$) such that
\[\int_{\Sigma_{\lambda_l}} F^2 \leq C \int_{\Omega_{\lambda_l}} |\nabla F|^2 + C\int_{\Omega_{\lambda_l}} F^2\]
for every smooth function $F$ on $\Omega_{\lambda_l}$.
\end{proposition}

\begin{proof}
Note that the hypersurface $\Sigma_{\lambda_l} = \partial\Omega_{\lambda_l}$ can be written as a radial graph with bounded slope. From this, it is easy to see that $\Omega_{\lambda_l}$ is bi-Lipschitz equivalent to the Euclidean unit ball, with constants that are independent of $l$. Hence, the assertion follows from the Sobolev trace theorem. 
\end{proof}

\begin{proposition}\label{est.W}
We have
\[\int_{\Sigma_{\lambda_l}} \max \{-W_{\lambda_l},0\} F^2 \leq o(1)\int_{\Omega_{\lambda_l}} |\nabla F|^2 + o(1)\int_{\Omega_{\lambda_l}} F^2\]
for every smooth function $F$ on $\Omega_{\lambda_l}$.
\end{proposition}

\begin{proof} 
The statement follows by combining \cite{Brendle}*{Corollary~A.7} (see also \cite{Fefferman-Phong}) with Corollary~\ref{negative.part.of.W}. 
\end{proof}

\begin{corollary}\label{sobolev.type.ineq.spinor.field}
For each section $s \in C^\infty(\Omega_{\lambda_l},\mathcal{E})$, we have 
\[\int_{\Omega_{\lambda_l}} |s|^2 \leq C \int_{\Omega_{\lambda_l}}  | \nabla^{\mathcal{E},q} s |^2 + C \int_U |s|^2,\] 
where $C$ is a constant that does not depend on $l$.
\end{corollary}

\begin{proof} 
By the triangle inequality, we can bound $|\nabla^{\mathcal{E}} s| \leq |\nabla^{\mathcal{E},q} s| + K \, |s|$, where $K$ is a constant that does not depend on $l$. This implies $(|\nabla^{\mathcal{E}} s| - K \, |s|)_+ \leq |\nabla^{\mathcal{E},q} s|$.

We now apply Proposition~\ref{sobolev.type.inequality} to the function $F = (\delta^2 + |s|^2)^{\frac{1}{2}}$, and send $\delta\to 0$. 
This gives 
\[\int_{\Omega_{\lambda_l}} |s|^2 \leq C \int_{\Omega_{\lambda_l}} (|\nabla^{\mathcal{E}} s| - K \, |s|)_+^2 + C \int_U |s|^2 .\] 
Putting these facts together, the assertion follows.
\end{proof}

\begin{corollary}\label{sobolev.trace.spinor.field}
For each section $s \in C^\infty(\Omega_{\lambda_l},\mathcal{E})$, we have 
\[\int_{\Sigma_{\lambda_l}} |s|^2 \leq C \int_{\Omega_{\lambda_l}}  | \nabla^{\mathcal{E},q} s |^2 + C \int_U |s|^2,\] 
where $C$ is a constant that does not depend on $l$.
\end{corollary}

\begin{proof} 
We apply Proposition~\ref{sobolev.trace} to the function $F = (\delta^2 + |s|^2)^{\frac{1}{2}}$, and send $\delta\to 0$. This gives 
\[\int_{\Sigma_{\lambda_l}} |s|^2 \leq C \int_{\Omega_{\lambda_l}} |\nabla^{\mathcal{E}} s|^2 + C \int_{\Omega_{\lambda_l}} |s|^2.\]
Hence, the assertion follows from Corollary~\ref{sobolev.type.ineq.spinor.field}.
\end{proof}

\begin{corollary}\label{est.W.spinor.field}
For each section $s \in C^\infty(\Omega_{\lambda_l},\mathcal{E})$, we have 
\[\int_{\Sigma_{\lambda_l}} \max \{-W_{\lambda_l},0\} |s|^2 \leq o(1) \int_{\Omega_{\lambda_l}}  | \nabla^{\mathcal{E},q} s |^2 + o(1) \int_U |s|^2 .\] 
\end{corollary}

\begin{proof} 
We apply Proposition~\ref{est.W} to the function $F = (\delta^2 + |s|^2)^{\frac{1}{2}}$, and send $\delta\to 0$. This gives 
\[\int_{\Sigma_{\lambda_l}} \max \{-W_{\lambda_l},0\} |s|^2 \leq o(1) \int_{\Omega_{\lambda_l}} |\nabla^{\mathcal{E}} s|^2 + o(1) \int_{\Omega_{\lambda_l}} |s|^2.\]
Hence, the assertion follows from Corollary~\ref{sobolev.type.ineq.spinor.field}.
\end{proof}

\begin{proposition} \label{prop:polytope} The following statements hold.
\begin{enumerate}[(a)]
\item There exists a $\nabla^{\mathcal{E},q}$-parallel section $s$ defined over $\Omega$. Considered as a homomorphism field, the section $s$ is invertible at each point in $\Omega$. 
\item Each of the boundary faces is totally geodesic with respect to $g$. 
\item Along each of the boundary faces, we have $q^{\rm tan} = 0$.
\end{enumerate}
\end{proposition}

Having established Proposition~\ref{prop:polytope}, the arguments in Section~\ref{immersion.into.Minkowski.space} imply that $(\Omega,g,q)$ can be isometrically immersed as a spacelike hypersurface in the Minkowski spacetime $\R^{n,1}$, with second fundamental form equal to $q$. 

In the remainder of this section, we will give the proof of Proposition~\ref{prop:polytope}. We first treat the odd-dimensional case, and then the even-dimensional case.

\subsection{Proof of Proposition~\ref{prop:polytope} in odd dimensions} 
Suppose that $n \geq 3$ is an odd integer. For each $l$, Proposition~\ref{proposition.existence.n.odd} implies the existence of a non-trivial section $s^{(l)} \in C^\infty(\Omega_{\lambda_l},\mathcal{E})$ such that 
\begin{align*} 
\int_{\Omega_{\lambda_l}}  | \nabla^{\mathcal{E},q} s^{(l)}|^2 
&\leq 
\tfrac{1}{2} \int_{\Omega_{\lambda_l}} (-\mu+|J|)  |s^{(l)}|^2 \\ 
&+ \tfrac{1}{2} \int_{\Sigma_{\lambda_l}} (|\tr(q) - q(\nu,\nu)| + \|dN^{(l)}\|_\tr - H) \, |s^{(l)}|^2. 
\end{align*}
Moreover, $\chi^{(l)} s^{(l)} = s^{(l)}$ at each point on the boundary $\Sigma_{\lambda_l}$. By assumption, $\mu-|J| \geq 0$. Using Corollary~\ref{est.W.spinor.field}, we obtain 
\begin{align*} 
\int_{\Sigma_{\lambda_l}} (|\tr(q) - q(\nu,\nu)| + & \|dN^{(l)}\|_\tr - H) \, |s^{(l)}|^2 \\ 
&\leq o(1) \int_{\Omega_{\lambda_l}}  | \nabla^{\mathcal{E},q} s^{(l)}|^2 + o(1) \int_U |s^{(l)}|^2. 
\end{align*}
Putting these facts together, we conclude that 
\[\int_{\Omega_{\lambda_l}}  | \nabla^{\mathcal{E},q} s^{(l)}|^2 \leq o(1) \int_U |s^{(l)}|^2\] 
if $l$ is sufficiently large. Using this estimate together with Corollary~\ref{sobolev.type.ineq.spinor.field}, we obtain 
\[\int_{\Omega_{\lambda_l}} |s^{(l)}|^2 \leq C \int_U |s^{(l)}|^2\] 
if $l$ is sufficiently large. Since $\int_{\Omega_{\lambda_l}} |s^{(l)}|^2 > 0$, it follows that $\int_U |s^{(l)}|^2 > 0$ if $l$ is sufficiently large. 

By scaling, we may arrange that $\int_U |s^{(l)}|^2 = 1$ if $l$ is sufficiently large. Then 
\begin{equation} \label{s_l.nearly.parallel.n.odd}
\int_{\Omega_{\lambda_l}}  | \nabla^{\mathcal{E},q} s^{(l)}|^2 \leq o(1) 
\end{equation}
and 
\[\int_{\Omega_{\lambda_l}} |s^{(l)}|^2 \leq C.\] 
Passing to a subsequence if necessary, the sequence $s^{(l)} $ converges weakly in $H_{\text{\rm loc}}^1(\Omega \setminus \partial \Omega)$. The limit $s$ is defined on the interior of $\Omega$ and satisfies $\int_U |s|^2 = 1$. Using \eqref{s_l.nearly.parallel.n.odd}, we obtain $\nabla^{\mathcal{E},q} s  = 0$. In particular, $s$ is smooth. 
Using $\nabla^{\mathcal{E},q}$-parallel transport along radial lines emanating from some point in the interior of $\Omega \subset \R^n$, we extend $s$ to a smooth section which is defined on an open set containing $\Omega$ and satisfies $\nabla^{\mathcal{E},q} s = 0$ at each point in $\Omega$. 

\begin{proposition}\label{proposition.convergence.of.s_l.n.odd}
We have
\[
\int_{\Sigma_{\lambda_l}} |s^{(l)} - s|^2 \to 0 \quad \text{ as } l \to \infty.
\]
\end{proposition}

\begin{proof} 
Using Corollary~\ref{sobolev.trace.spinor.field}, we obtain
\[\int_{\Sigma_{\lambda_l}} |s^{(l)} - s|^2 \leq C \int_{\Omega_{\lambda_l}}  |\nabla^{\mathcal{E},q} (s^{(l)}-s)|^2 + C \int_U  |s^{(l)}-s|^2.\] 
Since $\nabla^{\mathcal{E},q} s = 0$, it follows that 
\[\int_{\Sigma_{\lambda_l}}  |s^{(l)} - s|^2 \leq C \int_{\Omega_{\lambda_l}}   |\nabla^{\mathcal{E},q} s^{(l)} |^2 + C \int_U |s^{(l)}-s|^2.\] 
The assertion now follows from \eqref{s_l.nearly.parallel.n.odd} together with the fact that $s^{(l)} \to s$ in $L_{\text{\rm loc}}^2(\Omega \setminus \partial \Omega)$.
\end{proof}

\begin{corollary}\label{convergence.boundary.condition.n.odd}
We have 
\[
\int_{\Sigma_{\lambda_l}} |\chi^{(l)} s - s|^2 \to 0 \quad \text{ as }  l \to \infty.
\] 
Here, $\chi^{(l)}$ denotes the boundary operator on $\Sigma_{\lambda_l}$.
\end{corollary}

\begin{proof}
Recall that $s^{(l)}$ satisfies the boundary condition $\chi^{(l)} s^{(l)} = s^{(l)}$ at each point on $\Sigma_{\lambda_l}$. 
Hence, the assertion follows from Proposition~\ref{proposition.convergence.of.s_l.n.odd}. 
\end{proof}

For each $i_0 \in I$, Assumption~\ref{no.redundant.inequalities} implies that the set $\{u_{i_0} = 0\} \cap \bigcap_{i \in I \setminus \{i_0\}} \{u_i < 0\}$ is nonempty (see \cite{Brendle}*{Lemma~3.2}). 
Moreover, it follows from Corollary~\ref{convergence.boundary.condition.n.odd} that $s$ satisfies the boundary condition $\gamma(\nu_{i_0}) \circ s = s \circ \gamma_0(N_{i_0})$ on $\{u_{i_0} = 0\} \cap \bigcap_{i \in I \setminus \{i_0\}} \{u_i < 0\}$. 
 
\begin{lemma}
Let $\sigma \in S_0$. If $s\sigma$ vanishes at some point of $\Omega$, then it vanishes everywhere on  $\Omega$. 
\end{lemma}

\begin{proof}
The proof is analogous to the proof of Lemma~\ref{lemma.parallel.spinors.n.odd}. 
\end{proof}

As above, we define $L := \{\sigma \in S_0 \colon \text{\rm $s\sigma=0$ at each point in $\Omega$}\}$.

\begin{lemma}\label{lemma.L.invariant.poly.odd.dim}
The subspace $L \subset S_0$ is invariant under $\gamma_0(N_i)$ for each $i \in I$. 
\end{lemma}

\begin{proof}
The proof is analogous to the proof of Lemma~\ref{lemma.L.invariant.n.odd}.
\end{proof}

\begin{lemma} \label{maxrank.poly.odd.dim} 
We have $L=\{0\}$. 
Moreover, the section $s$, considered as a homomorphism field, is invertible on each point in $\Omega$. 
\end{lemma}

\begin{proof}
This follows from Lemma~\ref{lemma.L.invariant.poly.odd.dim} together with the fact that 
\begin{equation*}
\text{\rm span}\{N_i \colon i \in I\} = \R^n.
\qedhere
\end{equation*}
\end{proof}

\begin{lemma} \label{boundary.faces.odd.dim}
Let $i_0 \in I$. Along the boundary face $\{u_{i_0} = 0\} \cap \bigcap_{i \in I \setminus \{i_0\}} \{u_i < 0\}$, we have $h(X) = q(X)^{\text{\rm tan}} = 0$ whenever $X$ is a tangent vector to that boundary face.
\end{lemma}

\begin{proof} 
Recall that 
\[
\gamma(\nu_{i_0}) \circ s = s \circ \gamma_0(N_{i_0})
\] 
on the boundary face 
\[
\{u_{i_0} = 0\} \cap \bigcap_{i \in I \setminus \{i_0\}} \{u_i < 0\}.
\]
In the next step, we differentiate this identity in tangential direction. Let $p$ be a point on the boundary face $\{u_{i_0} = 0\} \cap \bigcap_{i \in I \setminus \{i_0\}} \{u_i < 0\}$, and let $X$ be a tangent vector to that boundary face. 
Arguing as in the proof of Lemma~\ref{lemma.boundary.behavior.n.odd}, we obtain 
\[\gamma(h(X)) \circ s + i \, \gamma(q(X)^{\text{\rm tan}}) \circ \gamma(\nu) \circ s = 0.\] 
Since $s \in \text{\rm Hom}(S_0,\mathcal{S}|_p)$ is invertible by Lemma~\ref{maxrank.poly.odd.dim}, it follows that 
\[\gamma(h(X)) + i \, \gamma(q(X)^{\text{\rm tan}}) \circ \gamma(\nu) = 0 \in \text{\rm End}(\mathcal{S}|_p).\] 
Note that $i \, \gamma(q(X)^{\text{\rm tan}}) \circ \gamma(\nu) \in \text{\rm End}(\mathcal{S}|_p)$ is self-adjoint, while $\gamma(h(X)) \in \text{\rm End}(\mathcal{S}|_p)$ is skew-adjoint. 
Thus, we conclude that 
\[
\gamma(h(X)) = 0 \in \text{\rm End}(\mathcal{S}|_p) \quad \text{ and } \quad \gamma(q(X)^{\text{\rm tan}}) \circ \gamma(\nu) = 0 \in \text{\rm End}(\mathcal{S}|_p).
\]
From this, the assertion follows.
\end{proof}

\subsection{Proof of Proposition~\ref{prop:polytope} in even dimensions} 
Suppose that $n \geq 2$ is an even integer. For each $l$, Corollary~\ref{corollary.existence.n.even} implies the existence of a non-trivial section $s^{(l)} \in H^1(\Omega_{\lambda_l},\mathcal{E})$ such that 
\begin{align*} 
\int_{\Omega_{\lambda_l}}  | \nabla^{\mathcal{E},q} s^{(l)}|^2 
&\leq 
\tfrac{1}{2} \int_{\Omega_{\lambda_l}} (-\mu+|J|)  |s^{(l)}|^2 \\ 
&\quad + \tfrac{1}{2} \int_{\Sigma_{\lambda_l}} (|\tr(q) - q(\nu,\nu)| + \|dN^{(l)}\|_\tr - H) \, |s^{(l)}|^2. 
\end{align*}
Moreover, for each $l$, we can find an element $\hat{t}_l \in S^1$ such that $\gamma(\nu) \circ s^{(l)} = s^{(l)} \circ \gamma_0(\hat{N}^{(l)})$ on $\Sigma_{\lambda_l}$, where the map $\hat{N}^{(l)} \colon \Sigma_{\lambda_l} \to S^n$ is defined by $\hat{N}^{(l)}(x) = \varphi(N^{(l)}(x),\hat{t}_l)$ for all $x \in \Sigma_{\lambda_l}$. By assumption, $\mu-|J| \geq 0$. Using Corollary~\ref{est.W.spinor.field}, we obtain 
\begin{align*} 
\int_{\Sigma_{\lambda_l}} (|\tr(q) - q(\nu,\nu)| & + \|dN^{(l)}\|_\tr - H) \, |s^{(l)}|^2 
\leq o(1) \int_{\Omega_{\lambda_l}}  | \nabla^{\mathcal{E},q} s^{(l)}|^2 + o(1) \int_U |s^{(l)}|^2. 
\end{align*}
Putting these facts together, we conclude that 
\[\int_{\Omega_{\lambda_l}}  | \nabla^{\mathcal{E},q} s^{(l)}|^2 \leq o(1) \int_U |s^{(l)}|^2\] 
if $l$ is sufficiently large. Using this estimate together with Corollary~\ref{sobolev.type.ineq.spinor.field}, we obtain 
\[\int_{\Omega_{\lambda_l}} |s^{(l)}|^2 \leq C \int_U |s^{(l)}|^2\] 
if $l$ is sufficiently large. Since $\int_{\Omega_{\lambda_l}} |s^{(l)}|^2 > 0$, it follows that $\int_U |s^{(l)}|^2 > 0$ if $l$ is sufficiently large. 

By scaling, we may arrange that $\int_U |s^{(l)}|^2 = 1$ if $l$ is sufficiently large. Then 
\begin{equation} \label{s_l.nearly.parallel.n.even}
\int_{\Omega_{\lambda_l}}  | \nabla^{\mathcal{E},q} s^{(l)}|^2 \leq o(1) 
\end{equation}
and $\int_{\Omega_{\lambda_l}} |s^{(l)}|^2 \leq C$.
After passing to a subsequence if necessary, we may assume that $\hat{t}_l \to \hat{t}$ for some $\hat{t} \in S^1$. 
Moreover, the sequence $s^{(l)} $ converges weakly in $H_{\text{\rm loc}}^1(\Omega \setminus \partial \Omega)$. 
The limit $s$ is defined on the interior of $\Omega$ and satisfies $\int_U |s|^2 = 1$. 
Using \eqref{s_l.nearly.parallel.n.even}, we obtain $\nabla^{\mathcal{E},q} s  = 0$. 
In particular, $s$ is smooth. 
Using $\nabla^{\mathcal{E},q}$-parallel transport along radial lines emanating from some point in the interior of $\Omega \subset \R^n$, we extend $s$ to a smooth section which is defined on an open set containing $\Omega$ and satisfies $\nabla^{\mathcal{E},q} s = 0$ at each point in $\Omega$. 

\begin{proposition}\label{proposition.convergence.of.s_l.n.even}
We have
\[
\int_{\Sigma_{\lambda_l}} |s^{(l)} - s|^2 \to 0\quad\text{ as } l\to\infty.
\]
\end{proposition}

\begin{proof} 
The proof is analogous to the proof of Proposition~\ref{proposition.convergence.of.s_l.n.odd}. 
\end{proof}

\begin{corollary}\label{convergence.boundary.condition.n.even}
We have 
\[
\int_{\Sigma_{\lambda_l}} |\gamma(\nu) \circ s - s \circ \gamma_0(\hat{N}^{(l)})|^2 \to 0 \quad\text{ as }  l\to\infty.
\]
\end{corollary}

\begin{proof}
Recall that $s^{(l)}$ satisfies the boundary condition $\gamma(\nu) \circ s^{(l)} = s^{(l)} \circ \gamma_0(\hat{N}^{(l)})$ on $\Sigma_{\lambda_l}$ at each point on $\Sigma_{\lambda_l}$. 
Hence, the assertion follows from Proposition~\ref{proposition.convergence.of.s_l.n.even}. 
\end{proof}

For each $i \in I$, we define $\hat{N}_i = \varphi(N_i,\hat{t}) \in S^n$. 
For each $i_0 \in I$, Assumption~\ref{no.redundant.inequalities} implies that the set $\{u_{i_0} = 0\} \cap \bigcap_{i \in I \setminus \{i_0\}} \{u_i < 0\}$ is nonempty (see \cite{Brendle}*{Lemma~3.2}). 
Moreover, it follows from Corollary~\ref{convergence.boundary.condition.n.even} that $s$ satisfies the boundary condition 
\[
\gamma(\nu_{i_0}) \circ s = s \circ \gamma_0(\hat{N}_{i_0})
\] 
on $\{u_{i_0} = 0\} \cap \bigcap_{i \in I \setminus \{i_0\}} \{u_i < 0\}$. 
By \cite{Brendle}*{Lemma~3.2}, the set $\{u_{i_0} = 0\} \cap \bigcap_{i \in I \setminus \{i_0\}} \{u_i < 0\}$ is a dense subset of $\Omega \cap \{u_{i_0}=0\}$. 
Thus, we conclude that 
\[
\gamma(\nu_{i_0}) \circ s = s \circ \gamma_0(\hat{N}_{i_0})
\] 
on $\Omega \cap \{u_{i_0}=0\}$. 

\begin{lemma} \label{lemma.comparison.of.N.and.hat.N.poly}
Suppose that $i_1,i_2 \in I$ satisfy $\Omega \cap \{u_{i_1}=0\} \cap \{u_{i_2}=0\} \neq \emptyset$. Then $\langle N_{i_1},N_{i_2} \rangle = \langle \hat{N}_{i_1},\hat{N}_{i_2} \rangle$.
\end{lemma}

\begin{proof} 
By assumption, we can find a point $p \in \Omega \cap \{u_{i_1}=0\} \cap \{u_{i_2}=0\}$. Then 
\[\gamma(\nu_{i_1}) \circ s = s \circ \gamma_0(\hat{N}_{i_1})\] 
and 
\[\gamma(\nu_{i_2}) \circ s = s \circ \gamma_0(\hat{N}_{i_2})\] 
at the point $p$. This implies 
\begin{align*} 
\langle \nu_{i_1},\nu_{i_2} \rangle_g \, |s|^2 
&= \tfrac{1}{2} \, \langle \gamma(\nu_{i_1}) \circ s,\gamma(\nu_{i_2}) \circ s \rangle + \tfrac{1}{2} \, \langle \gamma(\nu_{i_2}) \circ s,\gamma(\nu_{i_1}) \circ s \rangle \\ 
&= \tfrac{1}{2} \, \langle s \circ \gamma_0(\hat{N}_{i_1}),s \circ \gamma_0(\hat{N}_{i_2}) \rangle + \tfrac{1}{2} \, \langle s \circ \gamma_0(\hat{N}_{i_2}),s \circ \gamma_0(\hat{N}_{i_1}) \rangle \\ 
&= \langle \hat{N}_{i_1},\hat{N}_{i_2} \rangle \, |s|^2 
\end{align*} 
at the point $p$. 
On the other hand, the Matching Angle Hypothesis implies that 
\[
\langle \nu_{i_1},\nu_{i_2} \rangle_g \, |s|^2 = \langle N_{i_1},N_{i_2} \rangle \, |s|^2
\]
at the point $p$. Putting these facts together, we conclude that $\langle N_{i_1},N_{i_2} \rangle \, |s|^2 = \langle \hat{N}_{i_1},\hat{N}_{i_2} \rangle \, |s|^2$ at the point $p$. Since $|s|^2 > 0$ at each point in $\Omega$, the assertion follows.
\end{proof}

\begin{lemma} \label{lemma.hat.t.poly}
The map $\varphi(\cdot,\hat{t}) \colon S^{n-1} \to S^n$ is the standard inclusion $S^{n-1} \hookrightarrow S^n$. 
\end{lemma}

\begin{proof}
Suppose that the assertion is false. By Lemma~\ref{lemma.existence.of.varphi}, the map $\varphi(\cdot,\hat{t})$ is Lipschitz continuous with a Lipschitz constant strictly less than $1$. Moreover, it follows from Assumption~\ref{no.redundant.inequalities} that $N_{i_1} \neq N_{i_2}$ for $i_1 \neq i_2$. This implies 
\[|N_{i_1}-N_{i_2}|^2 > |\varphi(N_{i_1},\hat{t}) - \varphi(N_{i_2},\hat{t})|^2 = |\hat{N}_{i_1}-\hat{N}_{i_2}|^2\] 
for $i_1 \neq i_2$. This contradicts Lemma~\ref{lemma.comparison.of.N.and.hat.N.poly}.
\end{proof}

\begin{lemma}
Let $\sigma \in S_0$. If $s\sigma$ vanishes at some point of $\Omega$, then it vanishes everywhere on  $\Omega$. 
\end{lemma}

\begin{proof}
The proof is analogous to the proof of Lemma~\ref{lemma.parallel.spinors.n.odd}. 
\end{proof}

We again define $L := \{\sigma \in S_0 \colon \text{\rm $s\sigma=0$ at each point in $\Omega$}\}$.

\begin{lemma}\label{lemma.L.invariant.poly.even.dim}
The subspace $L \subset S_0$ is invariant under $\gamma_0(\hat{N}_i)$ for each $i \in I$. 
\end{lemma}

\begin{proof}
The proof is analogous to the proof of Lemma~\ref{lemma.L.invariant.n.odd}.
\end{proof}

\begin{lemma} \label{maxrank.poly.even.dim} 
We have $L=\{0\}$. Moreover, the section $s$, considered as a homomorphism field, is invertible on each point in $\Omega$. 
\end{lemma}

\begin{proof}
Note that $\text{\rm span}\{N_i \colon i \in I\} = \R^n$. 
Using Lemma~\ref{lemma.hat.t.poly}, it follows that 
\[
\text{\rm span}\{\hat{N}_i \colon i \in I\} = \R^n \subset \R^{n+1}.
\]
By Lemma~\ref{lemma.L.invariant.poly.even.dim}, $L$ is invariant under $\gamma_0(\xi)$ for all $\xi \in \R^n$. Since the volume element acts as a scalar multiple of the identity, we deduce that $L$ is invariant under $\gamma_0(\xi)$ for all $\xi \in \R^{n+1}$. Arguing as in the proof of Lemma~\ref{lemma.maxrank.n.odd}, the assertion follows.
\end{proof}

\begin{lemma} \label{boundary.faces.even.dim}
Let $i_0 \in I$. Along the boundary face $\{u_{i_0} = 0\} \cap \bigcap_{i \in I \setminus \{i_0\}} \{u_i < 0\}$, we have $h(X) = q(X)^{\text{\rm tan}} = 0$ whenever $X$ is a tangent vector to that boundary face.
\end{lemma}

\begin{proof} 
The proof is analogous to the proof of Lemma~\ref{boundary.faces.odd.dim}.
\end{proof}

\appendix

\section{A metric characterization of the ball} \label{sec:metric_char}

\begin{theorem} \label{thm:completeconvex} 
Let $(M,g)$ be a compact connected Riemannian manifold of dimension~$n\ge2$ with boundary such that the following conditions are satisfied: 
\begin{itemize}
  \item $g$ is flat.
  \item The second fundamental form of $\partial M$ is weakly positive definite.
  \item There exists a point in $\partial M$, where the second fundamental form is strictly positive definite.
\end{itemize}
Then $(M,g)$ is isometric to a compact convex domain in $\R^n$ with smooth boundary.
In particular, $M$ is diffeomorphic to a ball.
\end{theorem}

\begin{proof} 
For $\varepsilon >0$ small enough, we consider the smooth manifold 
\[M' := M \cup_{\sim}  (\partial M \times (-\varepsilon,\infty) ),\]
where $\sim$ identifies $\partial M \times (-\varepsilon,0] \subset \partial M \times (-\varepsilon,\infty)$ with an open neighborhood of $\partial M \subset M$ by the normal exponential map along $\partial M$ with respect to $g$.
We denote the metric on $\partial M$ induced by $g$ by $g_0$ and the second fundamental form of $\partial M$ (considered as a $(1,1)$-tensor field) by $h$.
The Gauss lemma, the Riccati equation for the second fundamental form and the vanishing of the curvature of $M$ imply that on $\partial M\times (-\varepsilon,0]$ the metric $g$ takes the form $g=dt \otimes dt+g_t$ where $g_t$ is the metric on $\partial M$ given by $g_t(X,Y) = g_0((\id+t h)^2X,Y)$.
Here $t\in (-\varepsilon,\infty)$ is the standard parameter.
Since $h$ is weakly positive definite, we can use this formula to extend $g$ to $\partial M\times [0,\infty)$ and hence to a smooth metric $g'$ on all of $M'$.
By \cite{Baer-Gauduchon-Moroianu}*{Theorem~7.2}, the manifold $(M',g')$ is flat.
It is easy to see that $(M',g')$ is complete.

Let $\pi: \tilde{M}' \to M'$ denote the universal covering of $M'$. We equip $\tilde{M}'$ with the metric $\tilde{g}' = \pi^* g'$. 
Since $(M',g')$ is complete and flat, it follows that $(\tilde{M}',\tilde{g}')$ is isometric to the Euclidean space $\R^n$. 

We define $\tilde{M} = \pi^{-1}(M)$. Then $\tilde{M}$ is a closed domain in $\tilde{M}'$ with smooth boundary. Since the embedding $M \hookrightarrow M'$ is a homotopy equivalence, it follows that $\tilde{M}'$ is connected. We distinguish two cases:
\medskip 

\textit{Case 1:} 
Suppose that $n \geq 3$. By assumption, the second fundamental form of $\partial M$ is weakly positive definite at each point on $\partial M$, and is strictly positive definite at some point $p \in \partial M$. We denote by $\Sigma$ the connected component of $\partial M$ that contains the point $p$. Since $n \geq 3$, the Gauss equations imply that the sectional curvature of $\Sigma$ is nonnegative at each point on $\Sigma$, and is strictly positive at the point $p \in \Sigma$. Using the Cheeger-Gromoll splitting theorem \cite{Cheeger-Gromoll}*{Theorem~3}, we conclude that every covering of $\Sigma$ is compact.

In the next step, we fix a point $\tilde{p} \in \partial \tilde{M}$ with $\pi(\tilde{p}) = p$. We denote by $\tilde{\Sigma}$ the connected component of $\partial \tilde{M}$ that contains the point $\tilde{p}$. Then $\tilde{\Sigma}$ is a covering of $\Sigma$. In view of the discussion above, $\tilde{\Sigma}$ is compact.

In the following, we identify $\tilde{M}'$ with $\R^n$. With this understood, we may view $\tilde{\Sigma}$ as a compact connected hypersurface in $\R^n$. By the main theorem in \cite{DoCarmo-Lima} (see also \cite{Sacksteder}), we may write $\tilde{\Sigma} = \partial K$, where $K \subset \R^n$ is a compact convex domain in $\R^n$ with smooth boundary. Moreover, $\tilde{\Sigma}$ is diffeomorphic to $S^{n-1}$ and $K$ is diffeomorphic to a ball $B^n$. At the point $\tilde{p}$, the outward-pointing unit normal vector to $K$ coincides with the outward-pointing unit normal vector to $\tilde{M}$. This implies $\text{\rm int}(K) \cap \text{\rm int}(\tilde{M}) \neq \emptyset$. Since $\tilde{M}$ is connected, it follows that $\tilde{M} \subset K$. If $K \setminus \tilde{M}$ is nonempty, we can find a point $\tilde{x}$ in the closure of $K \setminus \tilde{M}$ which has maximal distance from $\tilde{p}$. It is easy to see that $\tilde{x} \in \text{\rm int}(K) \cap \partial \tilde{M}$ and the second fundamental form of $\partial \tilde{M}$ at the point $\tilde{x}$ is negative definite, contrary to our assumption. Thus, we conclude that $\tilde{M} = K$. In particular, $\tilde{M}$ is diffeomorphic to a ball $B^n$.

If $M$ is not simply connected, then there exists an isometry of $\tilde{M}'$ which maps $\tilde{M}$ to itself and has no fixed points. 
This contradicts the Brouwer fixed point theorem. 
Thus, $M$ is simply connected, and so is $M'$. 
From this, the assertion follows.
\medskip

\textit{Case 2:} 
Suppose that $n=2$. 
In this case, the Euler characteristic of $M$ is strictly positive by the Gauss-Bonnet theorem. 
Since $M$ is connected, it follows that $M$ is diffeomorphic to a disk $B^2$. 
In particular, $M$ is simply connected, and so is $M'$. 
Consequently, $(M',g')$ is isometric to $\R^2$ and $(M,g)$ can be identified with a compact convex domain in $R^2$ with smooth boundary.
\end{proof}

\section{An auxiliary estimate}

\begin{proposition}\label{sobolev.type.ineq.fixed.domain}
Let $\Omega\subset\R^n$ be a compact connected Lipschitz domain. Assume that $U$ is a nonempty open subset of $\Omega$. Suppose that $K>0$ is a given positive real number. Then 
\[\int_\Omega F^2 \leq C\int_\Omega (|\nabla F| -KF)_+^2 + C\int_U F^2\] 
for every nonnegative function $F\in H^1(\Omega)$. Here, $C$ is a positive constant that depends on $\Omega$, $U$, and $K$, but not on $F$.
\end{proposition} 

\begin{proof}
Suppose the assertion is false. Then we can find a sequence of nonnegative functions $F_j\in H^1(\Omega)$ such that $\int_\Omega F_j^2 = 1$, $\int_U F_j^2 \to 0$, and
\[\int_\Omega (|\nabla F_j| -KF_j)_+^2 \to 0.\]
Using the triangle inequality, we obtain 
\[\int_\Omega |\nabla F_j|^2 \leq C.\]
After passing to a subsequence, the sequence $F_j$ converges in $L^2(\Omega)$ to a nonnegative function $F$. The function $F$ satisfies $\int_\Omega F^2 = 1$. Moreover, $F = 0$ almost everywhere in $U$. 
	
Let $\eta \colon \R^n \to [0,\infty)$ denote a smooth function which is supported in the unit ball and satisfies $\int_{\R^n} \eta = 1$. For each $\varepsilon > 0$, we define a smooth function $\eta_\varepsilon \colon \R^n \to [0,\infty)$ by 
\[\eta_{\varepsilon}(x) = \varepsilon^{-n} \, \eta(\varepsilon^{-1} x)\]
for all $x \in \R^n$. The function $\eta_\varepsilon$ is supported in a ball of radius $\varepsilon$ and $\int_{\R^n} \eta_\varepsilon = 1$. 

Let us fix an arbitrary point $p \in U$. 
Let $\varepsilon$ be a real number such that $0 < \varepsilon < d(p,\partial U)$. Then $d(p,\partial \Omega) > \varepsilon$. 
We denote by $\Omega_\varepsilon$ the connected component of the set 
\[
\{x \in \Omega \colon d(x,\partial \Omega) > \varepsilon\}
\]
containing $p$. 
Let $F_j^\varepsilon \colon \Omega_\varepsilon \to \R$ denote the convolution of $F_j$ with $\eta_\varepsilon$, and let $F^\varepsilon \colon \Omega_\varepsilon \to \R$ denote the convolution of $F$ with $\eta_\varepsilon$. Then $|\nabla F_j^\varepsilon| \leq \eta_\varepsilon * |\nabla F_j|$ at each point in $\Omega_\varepsilon$. This implies 
\[|\nabla F_j^\varepsilon| - K F_j^\varepsilon \leq (\eta_\varepsilon * |\nabla F_j|) - K \, (\eta_\varepsilon * F_j) \leq \eta_\varepsilon * (|\nabla F_j| - KF_j)_+\]
at each point in $\Omega_\varepsilon$. 
We now pass to the limit as $j \to \infty$, keeping $\varepsilon$ fixed. 
Since $F_j \to F$ in $L^2(\Omega)$, we know that $F_j^\varepsilon \to F^\varepsilon$ and $|\nabla F_j^\varepsilon| \to |\nabla F^\varepsilon|$ at each point in $\Omega_\varepsilon$. 
Moreover, since $(|\nabla F_j| - KF_j \big )_+ \to 0$ in $L^2(\Omega)$, it follows that $\eta_\varepsilon * (|\nabla F_j| - KF_j)_+ \to 0$ at each point in $\Omega_\varepsilon$. 
Putting these facts together, we conclude that $|\nabla F^\varepsilon| - K F^\varepsilon \leq 0$ at each point in $\Omega_\varepsilon$. 
On the other hand, we know that $F=0$ almost everywhere in $U$. Since $d(p,\partial U) > \varepsilon$, it follows that $F^\varepsilon$ vanishes at the point $p$. Since $\Omega_\varepsilon$ is connected, standard ODE arguments imply that $F^\varepsilon = 0$ at each point in $\Omega_\varepsilon$. 

Finally, we send $\varepsilon \to 0$. Since $\Omega$ is connected, we conclude that $F=0$ almost everywhere in $\Omega$. 
This contradicts the fact that $\int_\Omega F^2 = 1$.
\end{proof}

%

\end{document}